\documentclass[12pt,envcountsame]{svmult}
\date{2.9.2010} 
\usepackage{amsmath,amssymb,amsfonts,euscript,mathrsfs,multirow,amscd}

\newcommand{\pmat}[1]{\begin{bmatrix} #1 \end{bmatrix}}
\newcommand{\z}{\mathfrak z} 
\newcommand{\fp}{\mathfrak p} 
\newcommand{\fq}{\mathfrak q} 
 
\newcommand{\fk}{\mathfrak k} 
\newcommand{\fu}{\mathfrak u} 
\newcommand{\fz}{\mathfrak z} 
\newcommand{\fpsq}{\mathfrak{psq}} 
\newcommand{\fsq}{\mathfrak{sq}} 
\newcommand{\fsu}{\mathfrak{su}} 

\newcommand{\tr}{\mathop{{\rm tr}}\nolimits}

\usepackage{amsthm}

\renewcommand{\tilde}{\widetilde} 
\renewcommand{\hat}{\widehat} 
\newcommand{\eset}{\emptyset} 
\newcommand{\oline}{\overline} 
\newcommand{\cG}{\mathcal G} 
\newcommand{\cH}{\mathscr H} 
\newcommand{\cD}{\mathscr D} 
\newcommand{\cK}{\mathscr K} 
\newcommand{\sZ}{\mathscr Z} 
\newcommand{\sU}{\mathscr U} 
\newcommand{\sV}{\mathscr V}

\newcommand{\R}{\mathbb R}
\newcommand{\N}{\mathbb N}
\newcommand{\C}{\mathbb C}

\newcommand{\g}{\mathfrak g} 
\newcommand{\ft}{\mathfrak t} 
\newcommand{\h}{\mathfrak h} 
 
\newcommand{\fsl}{\mathfrak{sl}} 
\newcommand{\fup}{\mathfrak{up}} 

\newcommand{\dd}{\mathsf{d}}
\renewcommand{\:}{\colon}

\newcommand{\Hol}{\mathpzc{Hol}}
\newcommand{\ad}{\mathop{{\rm ad}}\nolimits}
\newcommand{\ev}{\mathop{{\rm ev}}\nolimits}

\newcommand{\subeq}{\subseteq} 

\newcommand{\germ}{\mathfrak}
\newcommand{\eev}{{\overline{0}}}
\newcommand{\ood}{{\overline{1}}}
\newcommand{\RES}{\mathrm{Res}}
\newcommand{\IND}{\mathrm{Ind}}
\newcommand{\CONE}{\mathpzc{Cone}}
\newcommand{\STAR}{$\star$}
\newcommand{\INT}{\mathpzc{Int}}
\newcommand{\II}{\mathrm{I}}
\newcommand{\JJ}{\mathrm{J}}
\newcommand{\eps}{\varepsilon}

\numberwithin{theorem}{subsection}

\title*{Lie supergroups, unitary representations, and invariant cones}
\author{Karl-Hermann Neeb \and Hadi Salmasian}
\institute{
K.--H. Neeb \at Department Mathematik, Friedrich--Alexander--Universit\"{a}t Erlangen--N\"{u}rnberg,
Bismarckstra\ss e $1\frac{1}{2}$, 91054, Erlangen, Germany, 
\email{neeb@mi.uni-erlangen.de}
\and 
H. Salmasian \at Department of Mathematics and Statistics, University of Ottawa, 585 King Edward Ave., Ottawa, ON K1N 6N5, Canada, \email{hsalmasi@uottawa.ca}
}
\date{July 12, 2010}
\DeclareFontFamily{OT1}{pzc}{}
\DeclareFontShape{OT1}{pzc}{m}{it}{<-> s * [1.100] pzcmi7t}{}
\DeclareMathAlphabet{\mathpzc}{OT1}{pzc}{m}{it}

\begin{document}

\maketitle

\section{Introduction}
The goal of this article is twofold. First,
it presents an application of the theory of invariant convex cones
of Lie algebras to the study of unitary representations of
Lie supergroups.
Second, it provides an exposition of recent
results of the second author on the classification of irreducible unitary 
representations of nilpotent 
Lie supergroups using the method of orbits.

In relation to the first goal, 
it is shown that there is a close connection between unitary 
representations of Lie supergroups and 
dissipative 
unitary representations of Lie groups (in the sense of \cite{Ne00}).  
It will be shown that  
for a large class of Lie supergroups  
the only irreducible unitary representations 
are
highest weight modules in a suitable sense.
This circle of ideas leads to explicit 
necessary conditions 
for determining when a Lie supergroup has faithful unitary representations.
These necessary conditions are then
used to analyze the situation
for simple and semisimple Lie supergroups.

Pertaining to the second goal,
the main results in \cite{salmasian} are explained
in a more reader friendly style. Complete proofs of the results are
given in \cite{salmasian}, and will not be repeated. However, wherever appropriate, 
ideas of the proofs are sketched.

\begin{acknowledgement}
K.--H. Neeb was supported by DFG-grant NE 413/7-1, Schwerpunktprogramm
``Darstellungstheorie''. H. Salmasian was supported by an NSERC Discovery Grant 
and an Alexander von Humboldt Fellowhip for Experienced Researchers.
\end{acknowledgement}

\section{Algebraic background}

We start by introducing the notation and stating several
facts which are 
used in this article. The reader is assumed to be familiar
with basics of the theory of superalgebras, and therefore this section
is rather terse.
For more detailed accounts of the subject
the reader is referred to  \cite{berezin}, \cite{kac}, or \cite{schbook}.

Let $\mathbb S$ be an arbitrary associative unital ring. A 
possibly nonassociative $\mathbb S$-algebra $\germ s$ is called a
\emph{superalgebra} if it is
$\mathbb Z_2$-graded, i.e., 
$
\germ s=\germ s_\eev\oplus\germ s_\ood
$
where $\germ s_i\germ s_j\subseteq\germ s_{i+j}$.
The degree of a homogeneous 
element $a\in\germ s$ is denoted by $|a|$.

A superalgebra
$\germ s$ is called  
\emph{supercommutative}
if 
\[
ab=(-1)^{|a|\cdot|b|}ba
\] 
for every two homogeneous elements $a,b\in\germ s$.

A \emph{Lie superalgebra} is a superalgebra whose multiplication, usually called its 
\emph{superbracket}, satisfies graded analogues of antisymmetry and 
the Jacobi identity. This means that if $A,B,C$ are homogeneous elements of a Lie superalgebra, 
then 
\begin{gather*}
[A,B]=-(-1)^{|A|\cdot |B|}[B,A]\\
\intertext{and}
(-1)^{|A|\cdot |C|}[A,[B,C]]+(-1)^{|B|\cdot |A|}[B,[C,A]]+(-1)^{|C|\cdot |B|}[C,[A,B]] =0.
\end{gather*}

Let $\mathbb K$ be a field and $\germ g$ be a Lie superalgebra over $\mathbb K$.
If $\germ h$ is a Lie subsuperalgebra of $\germ g$  
then $\mathscr Z_\germ g(\germ h)$ denotes
the \emph{supercommutant} of $\germ h$ in $\germ g$, i.e.,
\[
\mathscr Z_\germ g(\germ h)=\{\,X\in\germ g\ |\ [\germ h,X]=\{0\}\,\}.
\]
The \emph{center} of $\germ g$ is the supercommutant of $\germ g$ in 
$\germ g$ and 
is denoted by $\mathscr Z(\germ g)$. The \emph{universal enveloping algebra}
of $\germ g$, which is defined in \cite[Sec. 1.1.3]{kac}, is denoted by
$\mathscr U(\germ g)$. 
The group of $\mathbb K$-linear automorphisms of 
$\germ g$ is denoted by $\mathrm{Aut}(\germ g)$.
Finally, recall that the definitions of nilpotent and solvable Lie superalgebras are 
the same as the ones for Lie algebras (see \cite[Sec 1]{kac}).

\subsection{Centroid, derivations, and differential constants}
Let $\mathbb K$ be an arbitrary field and $\germ s$ be a finite dimensional
superalgebra over  $\mathbb K$. 
The \emph{multiplication algebra} of $\germ s$, denoted by
$\mathscr M(\germ s)$, is 
the associative unital superalgebra over $\mathbb K$ which is 
generated by the elements $\mathrm R_x$ and $\mathrm L_x$
of $\mathrm{End}_{\mathbb K}(\germ s)$, for all homogeneous $x\in\germ s$, where 
\[
\mathrm L_x(y)=xy
\textrm{\ \ and\ \ }
\mathrm R_x(y)=(-1)^{|x|\cdot|y|}yx
\textrm{\ \ for every homogeneous }y\in\germ s.
\]
As usual, the superbracket on $\mathrm{End}_\mathbb K(\germ s)$ is
defined by
\[
[A,B]=AB-(-1)^{|A|\cdot|B|}BA
\]
for homogeneous elements $A,B\in\mathrm{End}_\mathbb K(\germ s)$, and is then extended to
$\mathrm{End}_\mathbb K(\germ s)$ by linearity.
The \emph{centroid} of $\germ s$, denoted by $\mathscr C(\germ s)$, is
the supercommutant of $\mathscr M(\germ s)$ in $\mathrm{End}_\mathbb K(\germ s)$, i.e, 
\[
\mathscr C(\germ s)=\big\{\ A\in\mathrm{End}_\mathbb K(\germ s)\ \big|
\  [A,B]=0\textrm{ for every }B\in\mathscr M(\germ s)\ \big\}.
\] 
Obviously $\mathscr C(\germ s)$ is a unital associative superalgebra over $\mathbb K$.
If $\germ s^2=\germ s$ then $\mathscr C(\germ s)$ is supercommutative
(see \cite[Prop. 2.1]{cheng} for a proof).

If $s\in\{\overline 0,\overline 1\}$, a  \emph{homogeneous derivation of degree $s$} of $\germ s$ is an element $D\in\mathrm{End}_\mathbb K(\germ s)$
such that for every two homogeneous elements $a,b\in \germ s$,
\[
D(ab)=D(a)b+(-1)^{|a|\cdot s}aD(b).
\]
The subspace of $\mathrm{End}_\mathbb K(\germ s)$ which is spanned by homogeneous 
derivations of $\germ s$ is a Lie superalgebra over $\mathbb K$
and is denoted by $\mathrm{Der}_\mathbb K(\germ s)$. 
The \emph{ring of differential constants}, denoted by $\mathscr R(\germ s)$, is 
the supercommutant of $\mathrm{Der}_\mathbb K(\germ s)$ 
in $\mathscr C(\germ s)$.

Suppose that $\germ s$ is \emph{simple}, i.e., $\germ s^2\neq\{0\}$ and 
$\germ s$ does not have proper two-sided ideals.
By Schur's Lemma every nonzero homogeneous element of $\mathscr C(\germ s)$ is
invertible. Since $\germ s^2$ is always a  two-sided ideal, $\germ s^2=\germ s$ and therefore  
$\mathscr C(\germ s)$ is supercommutative. 
It follows that $\mathscr C(\germ s)_\ood=\{0\}$, $\mathscr C(\germ s)_\eev$ is a 
field, and  
$\mathscr R(\germ s)$ is a subfield of
$\mathscr C(\germ s)_\eev$ containing $\mathbb K$.

\subsection{Derivations of base extensions}
Let $\mathbb K$ be a field of characteristic zero and
$\mathbf \Lambda(n,\mathbb K)$ be
the \emph{Gra\ss mann superalgebra over $\mathbb K$ in $n$ indeterminates},
i.e., the associative unital superalgebra over $\mathbb K$ generated by odd elements $\xi_1,\ldots,\xi_n$
modulo the relations 
\[
\xi_i\xi_j+\xi_j\xi_i=0\textrm{\ \ for every }1\leq i,j\leq n.
\]

Let $\germ s$ be a superalgebra over $\mathbb K$.
The tensor product $\germ s\otimes_\mathbb K\mathbf \Lambda(n,\mathbb K)$ is a superalgebra
over $\mathbb K$. 
Note that since $\mathbf \Lambda(n,\mathbb K)$ is supercommutative,
if $\germ s$ is a Lie superalgebra then so is 
$\germ s\otimes_\mathbb K\mathbf \Lambda(n,\mathbb K)$.

It is proved in \cite[Prop. 7.1]{cheng} that
\begin{equation}
\label{deriva}
\mathrm{Der}_\mathbb K\big(\germ s\otimes_\mathbb K\mathbf \Lambda(n,\mathbb K)\big)=
\mathrm{Der}_\mathbb K(\germ s)\otimes_\mathbb K\mathbf \Lambda(n,\mathbb K)+
\mathscr C(\germ s)\otimes_\mathbb K\mathbf W(n,\mathbb K)
\end{equation}
where 
\[
\mathbf W(n,\mathbb K)=\mathrm{Der}_\mathbb K\big(\mathbf \Lambda(n,\mathbb K)\big).
\] 
The right hand side of 
\eqref{deriva} acts on $\germ s\otimes_\mathbb K\mathbf \Lambda(n,\mathbb K)$ 
via
\begin{gather*}
(D_\germ s\otimes_\mathbb K a)(X\otimes_\mathbb K b)=
(-1)^{|a|\cdot |X|}D_\germ s(X)\otimes_\mathbb K ab\\
\intertext{and}
(T\otimes_\mathbb K D_\mathbf \Lambda)(X\otimes_\mathbb K a)=(-1)^{|D_\mathbf\Lambda|\cdot |X|} T(X)\otimes_\mathbb K D_\mathbf \Lambda(a).
\end{gather*}
Note that the right hand side of \eqref{deriva} is indeed a direct sum of the two summands. This 
follows from the fact that
every element of 
$\mathscr C(\germ s)\otimes_\mathbb K \mathbf W(n,\mathbb K)$ vanishes 
on $\germ s\otimes_\mathbb K 1_{\mathbf \Lambda(n,\mathbb K)}$, whereas an element of 
$\mathrm{Der}_\mathbb K(\germ s)\otimes_\mathbb K\mathbf \Lambda(n,\mathbb K)$
which vanishes on $\germ s\otimes_\mathbb K 1_{\mathbf \Lambda(n,\mathbb K)}$ must be
zero.

\subsection{Cartan subsuperalgebras} Let $\mathbb K$ be a field of characteristic zero
and $\germ g$ be a finite dimensional Lie superalgebra over $\mathbb K$. 
A Lie subsuperalgebra of $\germ g$ which is 
nilpotent and self normalizing is called a \emph{Cartan subsuperalgebra}.

An important property 
of Cartan subsuperalgebras of $\germ g$ is that 
they are
uniquely determined by their intersections with $\germ g_\eev$. Our next goal is to state this
fact more formally.

For every subset $\mathpzc W_\eev$ of $\germ g_\eev$, let
\[
\mathscr N_\germ g(\mathpzc W_\eev)=
\big\{\, X\in\germ g\,\big|\,\textrm{ for every }W\in \mathpzc W_\eev,\,
\textrm{if }k\gg 0\textrm{ then }
\mathrm{ad}(W)^k(X)=0
\,\big\}.
\]
One can easily prove that $\mathscr N_\germ g(\mathpzc W_\eev)$ is indeed a subsuperalgebra of $\germ g$.
The next proposition is stated in \cite[Prop. 1]{schpaper} (see also \cite[Prop. 1]{penkovserganova}).
\begin{proposition}
\label{proppenkovserganova}
If $\germ h=\germ h_\eev\oplus\germ h_\ood$ is a Cartan subsuperalgebra of $\germ g$
then $\germ h_\eev$ is a Cartan subalgebra of $\germ g_\eev$. Conversely, if $\germ h_\eev$ is a 
Cartan subalgebra of $\germ g_\eev$ then $\mathscr N_\germ g(\germ h_\eev)$ is a Cartan subsuperalgebra of 
$\germ g$. The correspondence
\[
\germ h_\eev\longleftrightarrow\mathscr N_\germ g(\germ h_\eev) 
\]
is a bijection between Cartan subalgebras of $\germ g_\eev$ and 
Cartan subsuperalgebras of $\germ g$.  
\end{proposition}

\subsection{Compactly embedded subalgebras}
Let $\germ g$ be a finite dimensional Lie superalgebra over $\mathbb R$.
The group $\mathrm{Aut}(\germ g)$ is a (possibly disconnected) Lie subgroup of 
$\mathrm{GL}(\germ g)$, the group
of invertible elements of $\mathrm{End}_\mathbb R(\germ g)$. 
The subgroup of $\mathrm{Aut}(\germ g)$
generated by $e^{\mathrm{ad}(\germ g_\eev)}$ is denoted by $\mathrm{Inn}(\germ g)$.

If $\germ h_\eev$ is a Lie subalgebra of $\germ g_\eev$ then
$\mathrm{INN}_\germ g(\germ h_\eev)$
denotes the closure in $\mathrm{Aut}(\germ g)$ of the subgroup generated by 
$e^{\mathrm{ad}(\germ h_\eev)}$.
When $\mathrm{INN}_\germ g(\germ h_\eev)$ is compact $\germ h_\eev$ is said to be
\emph{compactly embedded
in $\germ g$}.

Cartan subalgebras of $\germ g_\eev$ which are compactly embedded in $\germ g$ 
are especially interesting because
they yield root decompositions of the complexification of $\germ g$.
The next proposition states this fact formally.
In the next proposition, let $\tau$ denote
the usual complex conjugation of elements of 
$\germ g^\mathbb C=\germ g\otimes_\mathbb R\mathbb C$, i.e., $\tau(X+iY)=X-iY$ for every
$X,Y\in\germ g$.

\begin{proposition}
\label{propsupercompactcartan}
Let $\germ t_\eev$ be a Cartan subalgebra of $\germ g_\eev$ 
which is compactly embedded in $\germ g$. 
Then the following statements hold.
\begin{description}[iii]
\item[\rm(i)] $\germ t_\eev$ is abelian.
\item[\rm(ii)] One can decompose $\germ g^\mathbb C$ as  
\begin{equation}
\label{rootdecomp}
\germ g^\mathbb C=\bigoplus_{\alpha\in\Delta}\germ g^{\mathbb C,\alpha}
\end{equation}
where
\[
\Delta=\big\{\,\alpha\in \germ t_\eev^*\ \big|\ \germ g^{\mathbb C,\alpha}\neq\{0\}\,\big\}
\]
and
\[
\germ g^{\mathbb C,\alpha}=\big\{\,X\in\germ g^\mathbb C\,\big|\,[H,X]=i\alpha(H)X\textrm{ for every }
H\in\germ t_\eev\,\big\}.
\]
\item[\rm(iii)]
If $\alpha\in\Delta$ then $-\alpha\in\Delta$ as well, and if 
$X\in\germ g^{\mathbb C,\alpha}$ then $\tau(X)\in\germ g^{\mathbb C,-\alpha}$.
\item[\rm(iv)] $\germ g_\eev=\germ t_\eev\oplus[\germ t_\eev,\germ g_\eev]$.
\end{description}
\end{proposition}

\begin{proof}
The proof of 
\cite[Theorem VII.2.2]{Ne00} can be adapted to prove Parts (i), (ii), and (iii).
Part (iv) can be proved using the fact that 
$\germ t_\eev=\mathscr Z_{\germ g_\eev}(\germ t_\eev)$ 
(see \cite[Chap. VII]{bourbaki}).
\end{proof}

More generally, if $\germ g_\eev$ has a Cartan subalgebra which is compactly embedded
in $\germ g$, then
any Cartan subalgebra
of $\germ g^\mathbb C$ yields a root decomposition. This is the
content of the following proposition.
\begin{proposition}
\label{prop-symmetryofrootsystem}
Assume that $\germ g_\eev$ has a Cartan subalgebra which is compactly
embedded in $\germ g$. If $\germ h^\mathbb C$ is an arbitrary Cartan subalgebra of 
$\germ g^\mathbb C$ then 
$\germ h_\eev^\mathbb C$ is abelian and there exists a root decomposition of $\germ g^\mathbb C$
associated to $\germ h^\mathbb C$, i.e., 
\[
\germ g^\mathbb C=
\bigoplus_{\alpha\in\Delta(\germ h^\mathbb C)}
\germ g^{\mathbb C,\alpha}
\]
where 
\[
\Delta(\germ h^\mathbb C)
=
\left\{
\alpha\in(\germ h_\eev^\mathbb C)^*\ |\ 
\germ g^{\mathbb C,\alpha}\neq \{0\}\right\}
\]
and 
\[
\germ g^{\mathbb C,\alpha}=\{\ X\in\germ g^\mathbb C\ |\ [H,X]=i\alpha(H)X\ \textrm{ for every }
H\in\germ h_\eev^\mathbb C\ \}.
\]
Moreover, $\Delta(\germ h^\mathbb C)=-\Delta(\germ h^\mathbb C)$.
\end{proposition}

\begin{proof}
Let $\germ t_\eev$ be a Cartan subalgebra of $\germ g_\eev$ which is compactly embedded in 
$\germ g$. Then $\germ t_\eev^\mathbb C$ is a Cartan subalgebra of $\germ g_\eev^\mathbb C$, and 
by Proposition \ref{proppenkovserganova} it corresponds to a Cartan subsuperalgebra 
$\germ t^\mathbb C$
of $\germ g^\mathbb C$. 
Proposition \ref{propsupercompactcartan} implies that there is a root decomposition 
of $\germ g^\mathbb C$ associated
to $\germ t^\mathbb C$, and if $\Delta(\germ t^\mathbb C)$ 
denotes the corresponding set of roots  
then 
$\Delta(\germ t^\mathbb C)=-\Delta(\germ t^\mathbb C)$.
It is known that any two Cartan subalgebras of $\germ g^\mathbb C_\eev$ are conjugate under 
inner automorphisms of $\germ g_\eev^\mathbb C$. Using 
Proposition \ref{proppenkovserganova}
one can show that any two Cartan subsuperalgebras of $\germ g^\mathbb C$ are conjugate under 
the group of $\mathbb C$-linear automorphisms of  
$\germ g^\mathbb C$ generated by $e^{\mathrm{ad}(\germ g_\eev^\mathbb C)}$.
By conjugacy, the root decomposition associated to $\germ t^\mathbb C$ 
turns into one associated to $\germ h^\mathbb C$.
\end{proof}

\subsection{Simple and semisimple Lie superalgebras}

The classification of finite dimensional complex simple Lie superalgebras and their real forms is known from \cite{kac} and \cite{serganova}. Every complex simple Lie superalgebra is isomorphic 
to one of the following types.
\begin{description}[iiiiii]

\item[(i)] A Lie superalgebra of \emph{classical} type, i.e., 
$\mathbf A(m|n)$ where $m,n>0$, 
$\mathbf B(m|n)$ where $m\geq 0$ and $n>0$, 
$\mathbf C(n)$ where $n>1$, 
$\mathbf D(m|n)$ where $m>1$ and $n>0$, 
$\mathbf G(3)$, $\mathbf F(4)$, $\mathbf D(2|1,\alpha)$ where 
$\alpha\in \mathbb C\,\backslash\{0,-1\}$, $\mathbf P(n)$ where
$n>1$, or $\mathbf Q(n)$ 
where $n>1$.
\item[(ii)] A Lie superalgebra of {\it Cartan type,} i.e., $\mathbf W(n)$ where $n\geq 3$,
$\mathbf S(n)$ where $n\geq 4$, $\mathbf{\tilde{S}}(n)$ where $n$ is even and 
$n\geq 4$, or 
$\mathbf H(n)$ where $n\geq 5$.
\item[(iii)] A complex simple Lie algebra.

\end{description}

Let $\germ s$ be a finite dimensional 
real simple Lie superalgebra with nontrivial odd part, i.e.,  
$\germ s_\ood\neq\{0\}$. 
Since $\mathscr C(\germ s)$ is a finite dimensional field
exension of $\mathbb R$, we have $\mathscr C(\germ s)=\mathbb R$ or $\mathscr C(\germ s)=\mathbb C$. If 
$\mathscr C(\germ s)=\mathbb C$, then  $\germ s$ is a complex simple 
Lie superalgebra which is considered as a real Lie superalgebra. 
If $\mathscr C(\germ s)=\mathbb R$, then $\germ s$ is a real form of
the complex simple Lie superalgebra  
$\germ s\otimes_\mathbb R\mathbb C$.
The classification of these real forms 
is summarized in Table 1 at the end of this article. 

A Lie superalgebra is called \emph{semisimple} if it has no nontrivial solvable ideals.
Semisimple Lie superalgebras are not necessarily direct sums of simple 
Lie superalgebras. 
In fact the structure theory of semisimple Lie superalgebras is rather complicated. 
The following statement can be obtained by a slight modification of the
arguments in \cite{cheng}.

\begin{theorem}
\label{classification-of-semisimple}
If a real Lie superalgebra $\germ g$ is semisimple then
there exist real simple Lie superalgebras $\germ s_1,\ldots,\germ s_k$ 
and nonnegative integers $n_1,\ldots,n_k$ such that
\[
\displaystyle
\bigoplus_{i=1}^k\big(\germ s_i\otimes_{\mathbb K_i}\mathbf \Lambda(n_i,{\mathbb K_i})\big)
\subseteq 
\germ g
\subseteq 
\bigoplus_{i=1}^k
\big(\mathrm{Der}_{\mathbb K_i}(\germ s_i)
\otimes_{\mathbb K_i} 
\mathbf \Lambda(n_i,{\mathbb K_i}) +
\mathbb L_i\otimes_{\mathbb K_i} \mathbf W(n_i,{\mathbb K_i})
\big)
\]
where 
$\mathbb K_i=\mathscr R(\germ s_i)$ and $\mathbb L_i=\mathscr C(\germ s_i)$
for every $1\leq i\leq k$.

\end{theorem}

\section{Geometric background}
 
Since we are interested in studying unitary representations from an analytic viewpoint,
we need to realize them as representations of Lie supergroups on 
$\mathbb Z_2$-graded Hilbert spaces. To this end, we first need to make precise
what we mean by Lie supergroups.

One can define Lie supergroups abstractly as \emph{group objects} in the
category of supermanifolds. To give sense to this definition,
one needs to define the category of supermanifolds. It will be seen below that
this can be done
by means of sheaves and ringed spaces.

Nevertheless, the above abstract definition of Lie supergroups is not 
well-suited for the study of unitary representations, and a
more explicit description of Lie supergroups is necessary. 
The aim of this section is to explain the latter 
description, which is based on the notion of Harish--Chandra pairs, and
to clarify the
relation between Harish--Chandra pairs and the categorical definition of Lie supergroups.

This section starts with a quick review of the theory of supermanifolds.
The reader 
who is not familiar with the basics of this subject and is interested in 
further detail
is referred to \cite{delignemorgan}, \cite{kostant}, \cite{leites}, 
\cite{manin}, and \cite{rajasupersym}.

We remind the reader that in the study of 
unitary representations only the simple point of view of
Harish--Chandra pairs
will be used. Therefore the reader may also skip the review of supergeometry and continue reading 
from Section~\ref{sectionsuperharishchandra}, where Harish--Chandra pairs are introduced.

\subsection{Supermanifolds}
Let $p$ and $q$ be nonnegative integers, and let $\mathscr O_{\mathbb R^p}$ denote the sheaf of smooth 
real valued functions on  $\mathbb R^p$. The \emph{smooth $(p|q)$-dimensional superspace} $\mathbb R^{p|q}$ is the 
ringed space $(\mathbb R^p,\mathscr O_{\mathbb R^{p|q}})$ where $\mathscr O_{\mathbb R^{p|q}}$ is the sheaf of 
smooth \emph{superfunctions} in $q$ odd coordinates.
The latter statement simply means that for every open $U\subseteq\mathbb R^p$ one has 
\[
\mathscr O_{\mathbb R^{p|q}}(U)=\mathscr O_{{\mathbb R}^p}(U)
\otimes_\mathbb R
\mathbf \Lambda(q,\mathbb R)\]
and the restriction maps of $\mathscr O_{\mathbb R^{p|q}}$ are obtained by
base extensions of the restriction maps of 
$\mathscr O_{{\mathbb R}^{p}}$. 

The ringed space 
$(\mathbb R^p,\mathscr O_{\mathbb R^{p|q}})$ is an object of the category 
$\mathsf{{Top}_{s-alg}}$
of topological spaces which are endowed with sheaves of associative unital
superalgebras over $\mathbb R$. 
If $\mathcal X=(\mathcal X_\circ,\mathscr O_\mathcal X)$ and $\mathcal Y=(\mathcal Y_\circ,\mathscr O_\mathcal Y)$
are objects in 
$\mathsf{{Top}_{s-alg}}$
then a morphism $\varphi:\mathcal X\to\mathcal Y$
is a pair $(\varphi_\circ,\varphi^\#)$ such that 
$\varphi_\circ:\mathcal X_\circ\to \mathcal Y_\circ$ is a continuous map and  
\[
\varphi^\#:\mathscr O_\mathcal Y\to(\varphi_\circ)_*\mathscr O_\mathcal X
\] 
is a morphism of sheaves of associative unital superalgebras over $\mathbb R$, where 
$(\varphi_\circ)_*\mathscr O_\mathcal X$ is the direct image 
\footnote{In \cite{delignemorgan} the authors define morphisms based on pullback.
Since pullback and direct image are adjoint functors, the definition 
of \cite{delignemorgan}
is equivalent to the definition given in this article, which is 
also used in \cite{leites}.} of $\mathscr O_\mathcal X$.

An object of $\mathsf{{Top}_{s-alg}}$ 
is called a \emph{supermanifold} of dimension $(p|q)$ if it is locally 
isomorphic to $\mathbb R^{p|q}$. 
supermanifolds constitute objects of a full subcategory of $\mathsf{{Top}_{s-alg}}$.

\subsection{Some basic constructions for supermanifolds}
If $\mathcal M=(\mathcal M_\circ,\mathscr O_\mathcal M)$ is a supermanifold 
of dimension $(p|q)$ then the nilpotent sections
of $\mathscr O_\mathcal M$ generate a sheaf of ideals $\mathscr I_\mathcal M$. Indeed the 
underlying space 
$\mathcal M_\circ$ 
is an ordinary smooth manifold
whose sheaf of smooth functions is 
$\mathscr O_\mathcal M/\mathscr I_\mathcal M$. 
One can also show that if $\mathcal M=(\mathcal M_\circ,\mathscr O_\mathcal M)$ and 
$\mathcal N=(\mathcal N_\circ,\mathscr O_\mathcal N)$ 
are two supermanifolds and
$\varphi:\mathcal M\to\mathcal N$ is a morphism 
then the map
$\varphi_\circ:\mathcal M_\circ\to \mathcal N_\circ$ is smooth
(see \cite[Sec. 2.1.5]{leites}).

Locally,
$\mathscr O_\mathcal M/\mathscr I_\mathcal M$ 
is isomorphic to $\mathscr O_{\mathbb R^p}$.
Therefore, if $U\subseteq\mathcal M_\circ$ is an open set, then 
for every section  
$f\in\mathscr O_\mathcal M(U)$ and every point $m\in U$ 
the value $f(m)$ is well defined. In this fashion, from any section $f$ one
obtains a smooth map 
\[
\tilde f:U\to\mathbb R.
\]
Nevertheless, because of the existence of nilpotent sections,
$f$ is not uniquely determined  by $\tilde f$.

Supermanifolds resemble ordinary manifolds in many ways. For example, 
one can prove the existence of  
finite direct products
in the category of supermanifolds.
Moreover, for 
a supermanifold $\mathcal M$ of dimension $(p|q)$
the sheaf $\mathrm{Der}_\mathbb R(\mathscr O_\mathcal M)$ of $\mathbb R$-linear derivations of the structural sheaf 
$\mathscr O_\mathcal M$  
is a locally free sheaf of $\mathscr O_\mathcal M$-modules of rank $(p|q)$.
Sections of the latter sheaf are called \emph{vector fields} of $\mathcal M$.
The space of vector fields is closed under the superbracket induced from 
$\mathrm{End}_\mathbb R(\mathscr O_\mathcal M)$.

If $\mathcal M=(\mathcal M_\circ,\mathscr O_\mathcal M)$ is a supermanifold and 
$m\in\mathcal M_\circ$, then there exists an obvious morphism 
\[
\delta_m:\mathbb R^{0|0}\to\mathcal M
\]
where $(\delta_m)_\circ:\mathbb R^0\to\mathcal M_\circ$ 
maps the unique point of $\mathbb R^{0}$ to $m$, and for every open set
$U\subseteq \mathcal M_\circ$ if $f\in\mathscr O_\mathcal M(U)$ then
\[
(\delta_m)^\#(f)=
\left\{
\begin{array}{ll}
\tilde f(m)\quad&\textrm{if }m\in U,\\
0&\textrm{otherwise.}
\end{array}
\right.
\]
Moreover, $\mathbb R^{0|0}$ is a terminal object in the category of supermanifolds.
Indeed for every supermanifold $\mathcal M=(\mathcal M_\circ,\mathscr O_\mathcal M)$
there exists a morphism 
\[
\kappa_\mathcal M:\mathcal M\to\mathbb R^{0|0}
\]
such that $(\kappa_\mathcal M)_\circ:\mathcal M_\circ\to\mathbb R^0$ maps every point of 
$\mathcal M_\circ$ to the unique point of $\mathbb R^0$ and for every 
$t\in\mathscr O_{\mathbb R^{0|0}}(\mathbb R^0)\simeq\mathbb R$ one has
$
(\kappa_\mathcal M)^\#(t)
=
t\cdot 1_\mathcal M
$.

\subsection{Lie supergroups and their Lie superalgebras}
Recall that by a \emph{Lie supergroup} we mean a group object in the 
category of supermanifolds. 
In other words, a supermanifold $\mathcal G=(\mathcal G_\circ,\mathscr O_\mathcal G)$ is a Lie supergroup if there exist morphisms
\begin{gather*}
\mu:\mathcal G\times\mathcal G\to\mathcal G\,,\ \,\varepsilon:\mathbb R^{0|0}\to\mathcal G\,,\ \ \textrm{and}\ \,
\iota:\mathcal G\to\mathcal G
\end{gather*} 
which satisfy the standard relations that describe associativity, existence of an identity element, and inversion. 
It follows that $\mathcal G_\circ$ is a Lie group whose multiplication
is given by $\mu_\circ:\mathcal G_\circ\times\mathcal G_\circ\to\mathcal G_\circ$.

To a Lie supergroup $\mathcal G$ one can associate a Lie superalgebra $\mathrm{Lie}(\mathcal G)$
which is the subspace of $\mathrm{Der}_\mathbb R(\mathscr O_\mathcal G)$ consisting of
left invariant
vector fields of $\mathcal G$.
The only subtle point in the definition of $\mathrm{Lie}(\mathcal G)$
is the definition of left invariant vector fields.
Left invariant vector fields can be defined in several ways. For example, in 
\cite{delignemorgan} the authors use the functor of points. We would like to mention 
a different method
which is also described in \cite{boyersanchez}.
For every $g\in\mathcal G_\circ$, one can define left translation morphisms 
$\lambda_g:\mathcal G\to\mathcal G$ by 
\[
\lambda_g=\mu\circ ((\delta_g\circ \kappa_\mathcal G)\times\mathrm{id}_\mathcal G)
\]
where 
$\mathrm{id}_\mathcal G:\mathcal G\to\mathcal G$ is the identity morphism. Similarly, one can define 
right translation morphisms 
\[
\rho_g=\mu\circ(\mathrm{id}_\mathcal G\times(\delta_g\circ\kappa_\mathcal G)).
\] 
A vector field $D$ is called
\emph{left invariant} if it commutes with left translation, i.e., 
\[
(\lambda_g)^\#\circ D=D\circ(\lambda_g)^\#.
\]
It is easily checked that $\mathrm{Lie}(\mathcal G)$, 
the space of left invariant vector fields of $\mathcal G$, 
is closed under
the super bracket of $\mathrm{Der}_\mathbb R(\mathscr O_\mathcal M)$.
Moreover, there is an action of $\mathcal G_\circ$ on 
$\mathrm{Lie}(\mathcal G)$ 
given by 
\begin{equation}
\label{actionofgnod}
D\mapsto (\rho_g)^\#\circ D\circ (\rho_{\iota_\circ(g)})^\#.
\end{equation}
Because of Part (ii) of Proposition \ref{fourpropertiesofg} below
it is natural to denote this action by $\mathrm{Ad}(g)$.
\begin{proposition}
\label{fourpropertiesofg}
For a Lie supergroup 
$\mathcal G=(\mathcal G_\circ,\mathscr O_\mathcal G)$ the following statements hold.
\begin{description}[iiiiii]
\item[\rm(i)] $\mathrm{Lie}(\mathcal G)=\mathrm{Lie}(\mathcal G)_\eev\oplus\mathrm{Lie}(\mathcal G)_\ood$ is a Lie superalgebra over $\mathbb R$.
\item[\rm(ii)] The action of $\mathcal G_\circ$ on $\mathrm{Lie}(\mathcal G)$ given by
\eqref{actionofgnod} yields a smooth homomorphism of Lie groups
\[
\mathrm{Ad}:\mathcal G_\circ\to\mathrm{GL}(\mathrm{Lie}(\mathcal G))
\]
such that $\mathrm{Ad}(\mathcal G_\circ)\subseteq\mathrm{End}_\mathbb R(\mathrm{Lie}(\mathcal G))_\eev$.
\item[\rm(iii)] $\mathrm{Lie}(\mathcal G)_\eev$ is the Lie algebra of $\mathcal G_\circ$ and if $\mathsf d(\mathrm{Ad})$ denotes the differential of the above 
map $\mathrm{Ad}$, then 
\[
\mathsf d(\mathrm{Ad})(X)(Y)=\mathrm{ad}(X)(Y)
\]
for every $X\in\mathrm{Lie}(\mathcal G)_\eev$ and every $Y\in \mathrm{Lie}(\mathcal G)$,
where
\[\mathrm{ad}(X)(Y)=[X,Y].\]

\end{description}

\end{proposition}

\subsection{Harish--Chandra pairs} 
\label{sectionsuperharishchandra}
Proposition \ref{fourpropertiesofg} states that to a Lie supergroup $\mathcal G$ one can associate
an ordered pair $(\mathcal G_\circ,\mathrm{Lie}(\mathcal G))$, where 
$\mathcal G_\circ$ is a real Lie group and $\mathrm{Lie}(\mathcal G)$ is a Lie superalgebra 
over $\mathbb R$, which satisfy certain properties.
Such an ordered pair is  a \emph{Harish--Chandra pair}.

\begin{definition}
A Harish--Chandra pair is a pair $(G,\germ g)$ consisting of a Lie group
$G$ and a Lie superalgebra $\germ g$
which satisfy the following properties.
\begin{description}[iiiiii]
\item[\rm(i)] $\germ g_\eev$ is the Lie algebra of $G$.
\item[\rm(ii)] $G$ acts on $\germ g$ smoothly by $\mathbb R$-linear automorphisms. 
\item[\rm(iii)] The differential of the action of $G$ on $\germ g$ is equal to the adjoint action of $\germ g_\eev$ on $\germ g$.
\end{description}
\end{definition} 
One can prove that 
\[
\mathcal G\mapsto(\mathcal G_\circ,\mathrm{Lie}(\mathcal G))
\]
is an equivalence of categories from the category of Lie supergroups to the category of  
Harish--Chandra pairs.
Under this equivalence of categories, a morphism $\psi:\mathcal G\to\mathcal H$ in the category of Lie supergroups corresponds to a 
pair $(\psi_\circ,\psi_\mathrm{Lie})$ where $\psi_\circ:\mathcal G_\circ\to\mathcal H_\circ$ is a homomorphism of Lie groups,
\[
\psi_\mathrm{Lie}:\mathrm{Lie}(\mathcal G)\to\mathrm{Lie}(\mathcal H)
\] 
is a homomorphism of Lie superalgebras, and  
\[\mathsf d\psi_\circ=\psi_\mathrm{Lie}\big|_{\mathrm{Lie}(\mathcal G)_\eev}.
\]

\begin{remark}
Using Harish--Chandra pairs one can study Lie supergroups and their representations without any reference to the structural sheaves.
In the rest of this article, Lie supergroups will always be realized as Harish--Chandra pairs.
\end{remark}

\section{Unitary representations}
According to \cite[Sec. 4.4]{delignemorgan} one can define a
finite dimensional super Hilbert space as a finite dimensional complex 
$\mathbb Z_2$-graded vector space which is endowed
with an even super Hermitian form. Nevertheless, since the even super Hermitian form is generally indefinite, 
in the infinite dimensional case one should 
address the issues of topological completeness and separability. For the purpose of studying unitary representations
it would be 
slightly more convenient to take an equivalent approach which is 
more straightforward, but less canonical.

\subsection{Super Hilbert spaces}
A \emph{super Hilbert space} is a $\mathbb Z_2$-graded 
complex Hilbert space $\mathscr H=\mathscr H_\eev\oplus\mathscr H_\ood$ 
such that $\mathscr H_\eev$ and $\mathscr H_\ood$ are mutually orthogonal closed subspaces.
If $\langle\cdot,\cdot\rangle$ denotes the inner product of $\mathscr H$, then for every two
homogeneous elements $v,w\in\mathscr H$ the 
even super Hermitian form $(v,w)$ of $\mathscr H$ is defined by
\[
(v,w)=
\left\{
\begin{array}{ll}
0& \textrm{if }v\textrm{ and }w \textrm{ have opposite parity,}\\
\langle v,w\rangle& \textrm{if }v \textrm{ and }w\textrm{ are even,}\\
i\langle v,w\rangle& \textrm{if }v \textrm{ and }w\textrm{ are odd.}\\
\end{array}
\right.
\]
One can check that $(\cdot,\cdot)$ satisfies the properties stated in \cite[Sec. 4.4]{delignemorgan}.
In this article the latter sesquilinear form will not be used. 

\subsection{The definition of a unitary representation}
In order to obtain an analytic theory of unitary representations of 
Lie supergroups
one should deal with the same sort of analytic difficulties that 
exist in the case of Lie groups. One of the main difficulties is that 
in general one cannot
define the differential of an infinite dimensional representation of a  
Lie group on the entire
representation space. However, one can always define the differential on
certain invariant dense subspaces, such as the
space of \emph{smooth vectors}. 

In the rest of this article, the reader 
is assumed to be familiar with classical 
results in the theory of unitary representations of 
Lie groups.  For a detailed and 
readable treatment of this subject see \cite{varalie}.

If $\mathscr H$ is a (possibly $\mathbb Z_2$-graded) complex Hilbert space, the
group of unitary operators of $\mathscr H$ is denoted by $\mathbf U(\mathscr H)$.
As usual, if $\pi:G\to\mathbf U(\mathscr H)$ is a unitary representation of a Lie group $G$, then
the space of smooth vectors (respectively, analytic vectors) of $(\pi,\mathscr H)$ is 
denoted by $\mathscr H^\infty$ (respectively, $\mathscr H^\omega$).

\begin{definition}
\label{defofunirep}
Let $(G,\germ g)$ be a Lie supergroup. A unitary representation 
of $(G,\germ g)$ is a triple $(\pi,\rho^\pi,\mathscr H)$ satisfying the following properties. 
\begin{description}[iiiiii]
\item[\rm(i)] $\mathscr H=\mathscr H_\eev\oplus\mathscr H_\ood$ is a super Hilbert space.
\item[\rm(ii)] $(\pi,\mathscr H)$ is a unitary representation of $G$ and
$\pi(g)\in\mathrm{End}_\mathbb C(\mathscr H)_\eev$ for every $g\in G$. 
\item[\rm(iii)] $\rho^\pi:\germ g\to\mathrm{End}_\mathbb C(\mathscr H^\infty)$ is an
$\mathbb R$-linear $\mathbb Z_2$-graded map, where
$\mathscr H^\infty$ denotes the space of smooth vectors of $(\pi,\mathscr H)$. Moreover, for every
$X,Y\in\germ g_\ood$,
\[
\rho^\pi(X)\rho^\pi(Y)+\rho^\pi(Y)\rho^\pi(X)=-i\rho^\pi([X,Y]).
\] 

\item[\rm(iv)] If $\mathsf d\pi$ denotes the differential of $\pi$ then $\rho^\pi(X)=\mathsf \dd\pi(X)$
for every $X\in\germ g_\eev$.
\item[\rm(v)] For every $X\in\germ g_\ood$ the operator $\rho^\pi(X)$ is symmetric, i.e., if
$v,w\in\mathscr H^\infty$ then 
\[
\langle\rho^\pi(X)v,w\rangle=\langle v,\rho^\pi(X)w\rangle.
\]
\item[\rm(vi)] For every $g\in G$ and every $X\in\germ g$, 
\[
\rho^\pi\big(\mathrm{Ad}(g)(X)\big)=\pi(g)\rho^\pi(X)\pi(g)^{-1}.
\]
\end{description}
\end{definition}
\begin{remark}
It is easy to see that 
by letting an element $X_\eev+X_\ood\,\in\,\germ g_\eev\oplus\germ g_\ood$ act on $\mathscr H^\infty$ as 
\[
\rho^\pi(X_\eev)+e^{\frac{\pi}{4}i}\rho^\pi(X_\ood)
\]
one obtains 
from $\rho^\pi$ a homomorphism
of Lie superalgebras from $\germ g$ into $\mathrm{End}_\mathbb C(\mathscr H^\infty)$.
\end{remark}
\begin{remark}
Subrepresentations, irreducibility, and unitary equivalence of unitary representations of Lie supergroups 
are defined similar to unitary representations of 
Lie groups (see \cite{rajaetal}). Note that 
in the definition of unitary equivalence, 
intertwining operators are assumed to preserve the grading. This means that in general a unitary representation
is not necessarily unitarily equivalent to the one obtained by parity change.
\end{remark}

\begin{lemma} \label{lem:4.2.1} For each $X \in \g$, the operator 
$\rho^\pi(X) \: \cH^\infty \to \cH^\infty$ is continuous with respect to the 
Fr\'echet topology on $\cH^\infty$. Moreover, the bilinear map 
\begin{equation}
  \label{eq:liealgact}
\g \times \cH^\infty \to \cH^\infty, \quad (X, v) \mapsto 
\rho^\pi(X)v 
\end{equation}
is continuous. 
\end{lemma} 

\begin{proof} Since $\g$ is finite dimensional, it suffices to 
show that each operator $\rho^\pi(X)$ is continuous. 
For $X \in \g_\eev$ , this follows from the definition of the 
Fr\'echet topology on $\cH^\infty$. 

For $X \in \g_\ood$, the operator 
$\rho^\pi(X)$ on $\cH^\infty$ is symmetric (see Definition \ref{defofunirep}(v) ), hence the graph of $\rho^\pi(X)$ is closed.
Now the Closed Graph Theorem for Fr\'echet spaces 
(see \cite[Thm.~2.15]{Ru73}) implies its continuity. 
\end{proof}

From now on we assume that $\mathscr H$ is separable. 
Although this assumption is not needed in Definition \ref{defofunirep},
it helps in avoiding technical conditions in various constructions, e.g., 
when induced representations are defined.
Note that  
if $(\pi,\rho^\pi,\mathscr H)$ is irreducible then $\mathscr H$ is separable.

In Definition \ref{defofunirep} the fact that 
$\mathscr H^\infty$ is chosen as the space of the representation of $\germ g$
is not a limitation. In fact it is shown in \cite[Prop. 2]{rajaetal} that in some sense any
reasonable choice of the space of the representation of $\germ g$,
i.e., one which is dense in $\mathscr H$ and satisfies natural invariance properties under the
actions on $G$ and $\germ g$,
would yield a definition
equivalent to the one given above.
This fact also plays a role in showing that
restriction and induction functors are well defined.
Another useful fact, which follows from \cite[Prop. 3]{rajaetal},
is that the space $\mathscr H^\omega$ of analytic vectors of 
$(\pi,\mathscr H)$ is invariant under $\rho^\pi(\germ g)$.

\subsection{Restriction and induction}
Suppose that 
$\mathcal G=(G,\germ g)$ is a Lie supergroup, and $\mathcal H=(H,\germ h)$ is a Lie subsupergroup of $\mathcal G$. 
Let $(\pi,\rho^\pi,\mathscr H)$ be a unitary representation of $\mathcal G$. A priori it is not clear 
how to restrict $(\pi,\rho^\pi,\mathscr H)$ to $\mathcal H$. The difficulty is that in general the space
of smooth vectors of the restriction of $(\pi,\mathscr H)$ to $H$ will be larger than $\mathscr H^\infty$.
To circumvent this issue one can use
\cite[Prop. 2]{rajaetal} to show that
the action of $\mathcal H$ on $\mathscr H^\infty$ determines a unique unitary representation
of $\mathcal H$ on $\mathscr H$. This representation
is called the restriction of $(\pi,\rho^\pi,\mathscr H)$ to $\mathcal H$, and is denoted by
\[
\RES_\mathcal H^\mathcal G(\pi,\rho^\pi,\mathscr H).
\]

Inducing from $\mathcal H$ to $\mathcal G$ is more delicate.
Let $(\sigma,\rho^\sigma,\mathscr K)$ be a unitary representation of $\mathcal H$.
The first step towards defining a representation $(\pi,\rho^\pi,\mathscr H)$ of $\mathcal G$ that is induced from
$(\sigma,\rho^\sigma,\mathscr K)$ is to identify the super Hilbert space $\mathscr H$.
By analogy with the case of Lie groups one expects the super Hilbert space 
$\mathscr H$ to be a space of $\mathscr K$-valued functions on $\mathcal G$
which satisfy an equivariance property with respect to the left regular action of $\mathcal H$.
One can then describe the action of $\mathcal G$ by formal relations, hoping that a unitary representation,
as defined in Definition \ref{defofunirep}, is obtained. This formal
approach leads to technical complications and it is not clear how to get around some of them.
Nevertheless,  
at least in the special case that the homogeneous super space $\mathcal H\backslash \mathcal G$ is 
purely even, i.e., when $\dim \germ g_\ood=\dim\germ h_\ood$, it is shown in \cite[Sec. 3]{rajaetal} that
the induced representation can be defined rigorously. In this article, only the special case when both $G$ and
$H$ are unimodular groups is used, and in this case the induced representation is defined as follows.
Since the homogeneous space $\mathcal H\backslash \mathcal G$ is purely even,
there is a natural isomorphism $\mathcal H\backslash\mathcal G\simeq H\backslash G$.
Choose an invariant measure
$\mu$ on $H\backslash G$, and let $\mathscr H$ be the
space of measurable functions $f:G\to\mathscr K$
which satisfy the following properties. 
\begin{description}[iiiiii]
\item[\rm(i)] $f(hg)=\sigma(h)f(g)$
 for every $g\in G$ and every $h\in H$.
\item[\rm(ii)] $\int_{H\backslash G}||f||^2d\mu<\infty$
\end{description}
The action of $G$ on $\mathscr H$ is the right regular representation, i.e., 
\begin{equation*}
\big(\pi(g)f\big)(g_1)=f(g_1g)\quad \textrm{ for every }\quad g,g_1\in G, 
\end{equation*}
and one can easily check that it is unitary with respect to the
standard
inner product of $\mathscr H$. The most natural way to 
define the action of an element $X\in\germ g_\ood$ on an element $f\in\mathscr H^\infty$
is via the formula
\begin{equation}
\label{actionofgg}
\big(\rho^\pi(X)f\big)(g)=\rho^\sigma\big(\mathrm{Ad}(g)(X)\big)\big(f(g)\big).
\end{equation}
It is known that every $f\in\mathscr H^\infty$ is a smooth function from $G$ to $\mathscr K$
and $f(g)\in\mathscr K^\infty$ for every $g\in G$ \cite[Th. 5.1]{poulsen}. Consequently, the 
right hand side of  \eqref{actionofgg} is well defined. However, a priori it is not obvious why for an
element $X\in\germ g_\ood$ the 
right hand side of  
\eqref{actionofgg} belongs to $\mathscr H^\infty$. One can prove the weaker statement that
$\rho^\pi(X)f\in\mathscr H$ using a trick which is based on the ideas used in \cite{rajaetal}. 
Since this trick sheds some light on the situation, 
it may be worthwhile to mention it.
One can prove that the operator $\rho^\pi(X)$ is essentially self-adjoint. Let 
$\overline{\rho^\pi(X)}$
denote the closure of $\rho^\pi(X)$. The operator $I+\overline{\rho^\pi(X)}^2$
has a bounded inverse whose domain is all of $\mathscr H$ 
(this follows for instance from \cite[Chap. X, Prop. 4.2]{conway}).
For every $f\in\mathscr H^\infty$,
\[
\rho^\pi(X)f=\overline{\rho^\pi(X)}f=\overline{\rho^\pi(X)}(I+\overline{\rho^\pi(X)}^2)^{-1}
(I+\overline{\rho^\pi(X)}^2)f.
\]
Using the spectral theory of self-adjoint operators one can
show that the operator
$\overline{\rho^\pi(X)}(I+\overline{\rho^\pi(X)}^2)^{-1}$ is bounded. Moreover,
\[
(I+\overline{\rho^\pi(X)}^2)f=\big(I-\frac{i}{2}\mathsf d\pi([X,X])\big)f\in\mathscr H^\infty.
\]
Finally, boundedness of $\overline{\rho^\pi(X)}(I+\overline{\rho^\pi(X)}^2)^{-1}$
implies that $\rho^\pi(X)f\in\mathscr H$.

To prove that indeed $\rho^\pi(X)f\in\mathscr H^\infty$ requires more effort.
This is proved in \cite[Sec. 3]{rajaetal} in an indirect way. The idea of the proof is to find a 
dense subspace $\mathscr B\subseteq\mathscr H^\infty$ such that
$\rho^\pi(\germ g)\mathscr B\subseteq\mathscr B$. 
As shown in \cite[Sec. 3]{rajaetal},
one can take $\mathscr B$ to be
the 
subspace of $\mathscr H^\infty$ consisting of functions from $G$ to $\mathscr K$ with compact 
support modulo $H$.
That $(\pi,\rho^\pi,\mathscr H)$
is well defined then follows from
\cite[Prop. 2]{rajaetal}.

The representation $(\pi,\rho^\pi,\mathscr H)$ induced from $(\sigma,\rho^\sigma,\mathscr K)$ is denoted by
\[
\IND_{\mathcal H}^{\mathcal G}(\sigma,\rho^\sigma,\mathscr K).
\]
It can be shown \cite[Prop. 3.2.1]{salmasian}
that induction may be done in stages, i.e., if $\mathcal H$ is a Lie subsupergroup of $\mathcal G$, $\mathcal K$ 
is a Lie subsupergroup of $\mathcal H$, and $(\sigma,\rho^\sigma,\mathscr K)$ is a unitary 
representation of $\mathcal K$, then 
\[
\IND_\mathcal H^\mathcal G\IND_\mathcal K^\mathcal H(\sigma,\rho^\sigma,\mathscr K)\simeq
\IND_\mathcal K^\mathcal G(\sigma,\rho^\sigma,\mathscr K).
\]

\section{Invariant cones in Lie algebras}
The goal of this section is to take a brief look at convex cones in 
finite dimensional real Lie algebras which are invariant 
under the adjoint action. A natural reduction to the case 
where the cone is pointed and generating leads to an interesting 
class of Lie algebras with a particular structure that will be 
discussed below.

A closed convex cone $C$ in a finite 
dimensional vector space $V$ is said to be 
\emph{pointed} if $C \cap - C = \{0\}$, i.e., if
$C$ contains no affine lines. It is said to be 
\emph{generating} if $C - C = V$ or equivalently if
$\INT(C)$ is nonempty, where $\INT(C)$ denotes the set of interior points of $C$. 
If $C$ is a cone in a finite dimensional vector space $V$ then 
$C^\star$ denotes the cone in $V^*$ consisting of all $\lambda\in V^*$ such that 
$\lambda(x)\geq 0$ for every $x\in C$.

\subsection{Pointed generating invariant cones}
Let $\germ g$ be a finite dimensional Lie algebra over $\mathbb R$.
A cone $ C \subseteq \germ g$ is called \emph{invariant} if it is closed, convex, and
invariant 
under $\mathrm{Inn}(\germ g)$.

Suppose that $C$ is an invariant cone in  $\germ g$ and set
$\mathrm H(C) = C \cap - C$ and $\germ g(C)=C - C$.
The subspaces $\mathrm H(C)$ and $\germ g(C)$ are ideals of $\germ g$
and
$C/\mathrm H(C)$ is a pointed generating invariant cone in the quotient 
Lie algebra $\germ g(C)/\mathrm H(C)$. The main concern of the theory of 
invariant cones is to understand the situation when 
$C$ is pointed and generating. 

The existence of pointed generating invariant cones in a Lie algebra has the following
simple but useful consequence.
\begin{lemma}
\label{lemma-abelianideal}
Let $C$ be a pointed generating invariant cone in $\germ g$. If $\germ a$ is an 
abelian ideal of $\germ g$ then $\germ a\subseteq\mathscr Z(\germ g)$.
\end{lemma}
\begin{proof}
If $X\in\INT(C)$, then $C\supseteq e^{\mathrm{ad}(\germ a)}X=X+[\germ a,X]$. 
Since $C$ contains no affine lines, $[\germ a,X]=\{0\}$. Since $X\in\INT(C)$
is arbitrary, $\germ a\subseteq\mathscr Z(\germ g)$.
\end{proof}
To study invariant cones further, we need the following lemma.

\begin{lemma}
\label{lemma-haarmeasure} Let $V$ be a finite dimensional vector space, 
$S \subseteq V$ be a convex subset, and  
$K \subseteq \mathrm{GL}(V)$ be a subgroup which leaves $S$ invariant.
Suppose that the closure of $K$ in $\mathrm{GL}(V)$ is compact.
If $S$ is open or closed, then 
it contains $K$-fixed points. 
\end{lemma}

\begin{proof} Let $\overline{K}$ be the closure of $K$,
and $\mu_{\overline{K}}$ be a normalized Haar measure on $\overline{K}$. For every 
$v \in S$, the point $v_\circ=\int_{\overline{K}} (k\cdot v)\, d\mu_{\overline{K}}(k)$ is $K$-fixed, and 
it is easily verified that $v_\circ\in S$. 
 \end{proof}

The preceding lemma has the following interesting consequence for 
invariant cones. 

\begin{lemma}\label{lemma-compactness}
Let $C\subseteq\germ g$ be a pointed generating invariant cone.
Then a subalgebra 
$\germ k \subseteq \germ g$ is compactly embedded in $\germ g$ if and only if 
$
\mathscr Z_\germ g(\germ k) \cap \INT(C) \neq\emptyset$.
\end{lemma}

\begin{proof}
If $\germ k \subseteq \germ g$ is compactly embedded in $\germ g$
then Lemma~\ref{lemma-haarmeasure} implies that 
$\INT(C)$ contains fixed points for $\mathrm{Inn}_\germ g(\germ k)$, i.e., 
\begin{equation}
  \label{eq:intersect}
\mathscr Z_\germ g(\germ k)\cap\INT(C)\neq\emptyset.
\end{equation}
Conversely, if 
$\mathscr Z_\germ g(\germ k) \cap \INT(C) \neq\emptyset$
then set $K=\mathrm{Inn}_\germ g(\germ k)$ and observe that 
$K$ is a subgroup of $\mathrm{Inn}(\germ g)$ 
with a fixed point 
$X_0\in \INT(C)$. The set $C \cap (X_0 - C)$ is a compact $K$-invariant 
subset of $\germ g$ with interior points. This implies that $K$ 
is bounded in $\mathrm{GL}(\germ g)$ and therefore it has compact closure
in $\mathrm{Aut}(\germ g)$. 
\end{proof}

\subsection{Compactly embedded Cartan subalgebras}

Let $\germ g$ be a finite dimensional Lie algebra over $\mathbb R$. 
Our next goal is to show that the existence of a pointed generating invariant cone in $\germ g$
implies that $\germ g$ has compactly embedded Cartan subalgebras.
The next lemma shows how such a Cartan subalgebra can be obtained explicitly.

\begin{lemma}
\label{lemma-compact-cartan}
Let $C\subseteq\germ g$ be a pointed generating invariant cone. Suppose
that 
$Y\in\INT(C)$ is a regular element of $\germ g$, i.e., 
the subspace
\[
\mathscr N_\germ g(Y)=\bigcup_n \ker(\mathrm{ad}(Y)^n)
\] 
has minimal dimension. If 
$\germ t=\ker(\mathrm{ad}(Y))$, then $\germ t$ is a Cartan subalgebra of $\germ g$ which is 
compactly embedded in $\germ g$.
\end{lemma}
\begin{proof}
For any such $Y$, the subspace 
$\germ t=\mathscr N_\germ g(Y)$ is a Cartan subalgebra of $\germ g$
(see \cite[Chap. VII]{bourbaki}). 
Since $Y\in\mathscr Z_\germ g(\mathscr Z_\germ g(\mathbb R Y))$, Lemma \ref{lemma-compactness} 
implies that
$\mathscr Z_\germ g(\mathbb R Y)$
is compactly embedded in $\germ g$. It follows immediately that
$\mathbb RY$ is compactly embedded in $\germ g$. 
Therefore the endomorphism 
$\mathrm{ad}(Y):\germ g\to\germ g$ is semisimple and
\[
\germ t=\ker(\mathrm{ad}(Y))=\mathscr Z_\germ g(Y)
\]
from which it follows that $\germ t$ is compactly embedded in $\germ g$. 
\end{proof}

\begin{remark}
It is known that the set 
of regular elements of $\germ g$
is dense (see \cite[Chap. VII]{bourbaki}). 
Since $\INT(C)\neq\emptyset$, the intersection of 
$\INT(C)$ with the set of regular elements of $\germ g$ is nonempty.
\end{remark}

\subsection{Characterization of Lie algebras with invariant cones}
\label{section-characterization}
The material in this section is 
meant to shed light on the connection between invariant cones and
Hermitian Lie algebras. The reader is assumed to be familiar
with the classification of real semisimple Lie algebras.

The study of invariant cones in finite-dimensional Lie algebras 
was initiated by B.~Kostant, I.~E.~Segal and E.~B.~Vinberg \cite{Se76}, 
\cite{Vin80}. A structure theory of invariant cones in general 
finite dimensional Lie algebras was developed by Hilgert and Hofmann 
in \cite{HiHo89}. 
The characterization of those finite dimensional Lie algebras 
containing pointed generating invariant cones was obtained 
in \cite{Ne94} in terms of certain symplectic modules called 
of convex type, whose classification can be found in  \cite{Neu00}. 
A self-contained exposition of this theory is available in \cite{Ne00},
where the Lie algebras $\germ g$ for which there exist pointed generating invariant 
cones in
$\germ g \oplus \mathbb R$ are called \emph{admissible}.

\begin{example}\,(\textit{cf.} \cite{Vin80})
Suppose that $\germ g$ is a real simple Lie algebra 
with a Cartan decomposition $\germ g = \germ k \oplus \germ p$. 
Since $\germ p$ is a simple nontrivial $\germ k$-module, 
$\mathscr Z_\germ g(\germ k) = \mathscr Z(\germ k)$. 
If $C$ is a pointed generating invariant cone in 
$\germ g$, then from Lemma \ref{lemma-compactness} it follows that 
\[ 
\INT(C) \cap \mathscr Z(\germ k) \neq\emptyset.
\] 
In particular $\mathscr Z(\germ k) \neq\{0\}$, i.e., 
$\germ g$ is \emph{Hermitian}. Conversely, 
assume that $\germ g$ is Hermitian and $0 \neq Z \in \mathscr Z(\germ k)$. 
If $(\cdot,\cdot)$ denotes the Killing form of $\germ g$, then from
the Cartan decomposition $\mathrm{Inn}(\germ g) = \mathrm{Inn}(\germ k)e^{\mathrm{ad}( \germ p)}$ 
it follows that
\[ (\mathrm{Inn}(\germ g)Z, Z) 
= (e^{\mathrm{ad} (\germ p)}Z,Z). 
\] 
If $P\in\germ p$ then $(e^{\mathrm{ad}(P)}Z,Z)<0$ because
\[
(e^{\mathrm{ad}(P)}Z,Z)=\sum_{n=0}^\infty\frac{(\mathrm{ad}(P)^{2n}(Z),Z)}{(2n)!}
\]
and the linear transformations $\mathrm{ad}(P)^{2n}:\germ k\to\germ k$ are
positive definite with respect to $(\cdot,\cdot)$. 
It follows that $\mathrm{Inn}(\germ g)Z$ lies in a proper invariant 
cone $C \subseteq \germ g$. Since $\germ g$ is simple, 
$C$ is pointed and generating. 
\end{example}

A slight refinement of the above arguments shows that 
a reductive Lie algebra $\germ g$ is admissible if and only if 
$\mathscr Z_\germ g(\mathscr Z(\germ k)) = \germ k$ holds for a maximal 
compactly embedded subalgebra $\germ k$ of $\germ g$. 
Lie algebras satisfying this property are called {\it quasihermitian}. 
This is equivalent to 
all simple ideals of $\germ g$ being either compact or Hermitian. 
A reductive admissible Lie algebra contains pointed generating 
invariant cones if and only if it is not compact semisimple. 
This clarifies the structure of reductive Lie algebras with 
invariant cones. 

Below we shall need the following lemma.

\begin{lemma} \label{lem:qherm} Let $\g$ be a quasihermitian Lie algebra, 
$\fk \subeq \g$ be a maximal compactly embedded subalgebra of $\germ g$, and  
and $p_\fz \: \g \to\mathscr Z(\fk)$ be the fixed point projection 
for the compact group $e^{\ad \fk}$. 
Then every closed invariant convex subset $C \subeq \g$ 
satisfies $p_\z(C) = C \cap \mathscr Z(\fk)$. 
\end{lemma}

\begin{proof} Let $\fp \subeq \g$ be a $\fk$-invariant 
complement and recall that $\g$ is said to be quasihermitian if 
$\fk = \mathscr Z_\g(\mathscr Z(\fk))$. This condition implies in particular 
that $\fp$ contains no non-zero trivial $\fk$-submodule, so that 
$\mathscr Z_\g(\fk) = \mathscr Z(\fk)$. 
The assertion now follows from the proof of 
Lemma~\ref{lemma-haarmeasure}. 
\end{proof}

In the case of an arbitrary Lie algebra $\germ g$ having a pointed 
generating invariant cone, 
one can use Lemma \ref{lemma-abelianideal} to show that the 
maximal nilpotent ideal $\germ n$ of $\germ g$ 
is two-step nilpotent, i.e., a generalized 
Heisenberg algebra. Moreover, $\germ n$
clearly contains $\mathscr Z(\germ g)$, which is contained 
in any compactly 
embedded Cartan subalgebra $\germ t$ of $\germ g$. 
Let $\germ a \subseteq \germ t$ be a complement to $\mathscr Z(\germ g)$ 
and $\germ s$ be a $\germ t$-invariant Levi complement to $\germ n$ in 
$\germ g$ (which always exists), and
set $\germ l = \germ a \oplus \germ s$. Then $\germ l$ is reductive, 
$\germ g =\germ l\ltimes  \germ n$, and 
$\germ l$ is an admissible reductive Lie algebra  
(see \cite[Prop.~VII.1.9]{Ne00}). At this point the 
structure of $\germ n$ and $\germ l$ is quite clear. 
However, to derive a classification of Lie algebras with invariant cones
from this semidirect decomposition, 
one has to analyze the possibilities 
for the $\germ l$-module structure on $\germ n$ in some detail.
This is done in 
\cite{Ne94} and \cite{Neu00}. 

\section{Unitary representations and invariant cones}

A Lie supergroup $\mathcal G=(G,\germ g)$ is called \emph{\STAR -reduced} if 
for every nonzero $X\in\germ g$ there exists a unitary representation $(\pi,\rho^\pi,\mathcal H)$ of 
$\mathcal G$ such that $\rho^\pi(X)\neq 0$. Note that when $\germ g$ is simple, $\mathcal G$
is \STAR -reduced if and only if it has a nontrivial unitary representation.
In this section we study properties of \STAR -reduced Lie supergroups via methods based on the theory
of
invariant cones. We obtain necessary conditions for
a Lie supergroup $\mathcal G$ to be \STAR -reduced.
It turns out that these necessary conditions
are strong enough for the classification of \STAR -reduced simple Lie supergroups.

Let $\mathcal G=(G,\germ g)$ be an arbitrary Lie supergroup,
and let $(\pi,\rho^\pi,\mathscr H)$ be a unitary representation of 
$\mathcal G$. Fix an element $X\in\germ g_\ood$.
From
\[
\rho^\pi([X,X])=i[\rho^\pi(X),\rho^\pi(X)]=2i\rho^\pi(X)^2
\]
and the fact that the operator $\rho^\pi(X)$ is symmetric
it follows that 
\[
\langle i\rho^\pi([X,X])v,v\rangle\leq 0
\textrm{ for every }v\in\mathscr H^\infty.
\]
Let $\CONE(\mathcal G)$ denote the invariant cone in $\germ g_\eev$ which is 
generated by elements of the form $[X,X]$ where 
$X\in\germ g_\ood$. Linearity of $\rho^\pi$ implies that 
\begin{equation}
  \label{eq:dissi}
\langle i\rho^\pi(Y)v,v\rangle\leq 0\textrm{ for every }v\in\mathscr H^\infty
\textrm{ and every }Y\in\CONE(\mathcal G).
\end{equation}
This means that $\pi$ is $\CONE(\mathcal G)$-dissipative 
in the sense of \cite{Ne00}.  

\subsection{Properties of \STAR -reduced Lie supergroups }
\label{section-properties}
Unlike Lie groups, 
which are known to have faithful unitary representations, certain
Lie supergroups do not have such representations. The next proposition,
which is given in \cite[Lem. 4.1.1]{salmasian},
shows how this can happen. 
The proof of this proposition
is based on the fact that for every $X\in\germ g_1$,
the spectrum of $-i\rho^\pi([X,X])$ is nonnegative, so that a 
sum of such operators vanishes if and only if all summands vanish.

\begin{proposition}
\label{reducedlemma}
Let $(\pi,\rho^\pi,\mathscr H)$ be a unitary representation of $\mathcal G=(G,\germ g)$.
Suppose that 
elements
$
X_1,\ldots ,X_m\in\germ g_\ood
$ 
satisfy 
\[
[X_1,X_1]+\cdots+[X_m,X_m]=0.
\]
Then $\rho^\pi(X_1)=\cdots=\rho^\pi(X_m)=0$.
\end{proposition}

The next proposition provides
necessary conditions for a Lie supergroup to be \STAR -reduced.

\begin{proposition} 
\label{propertiesofstarreduced}
If $\mathcal G=(G,\germ g)$ is \STAR -reduced, 
then the following statements hold.
\begin{description}[iiiiii]
\item[\rm(i)] $\CONE(\mathcal G)$ is pointed.
\item[\rm(ii)] 
For every $\lambda\in\INT(\CONE(\mathcal G)^\star)$, the symmetric bilinear form 
\[
\mathsf\Omega_\lambda:\germ g_\eev\times\germ g_\eev\to\mathbb R 
\quad \mbox{ defined by } \quad 
\mathsf\Omega_\lambda(X,Y)=\lambda([X,Y])
\]
is positive definite .
\item[\rm(iii)] Let $\germ k_\eev$ be a Lie subalgebra of $\germ g_\eev$. If $\germ k_\eev$ is compactly embedded in
$\germ g_\eev$, then $\germ k_\eev$ is compactly embedded in $\germ g$ as well.
\item[\rm(iv)] If $\germ g_\eev=[\germ g_\ood,\germ g_\ood]$ then
$\germ g_\eev$ has a Cartan subalgebra which is compactly embedded in $\germ g$. 
\item[\rm(v)] Assume that there exists a Cartan subalgebra $\germ h_\eev$ of 
$\germ g_\eev$ which is compactly embedded in $\germ g$. Let 
$
\mathbf p:\germ g_\eev\to\germ h_\eev
$
be the projection map corresponding to the decomposition 
\[
\germ g_\eev=\germ h_\eev\oplus[\germ h_\eev,\germ g_\eev]
\] 
(see Proposition \ref{propsupercompactcartan})
and $\mathbf p^*:\germ h_\eev^*\to\germ g_\eev^*$ be
the corresponding dual map.
Then
\[
\INT(\CONE(\mathcal G)^\star)\cap\mathbf p^*(\germ h_\eev^*)\neq \emptyset.
\]
\end{description}
\end{proposition}

\begin{proof}
(i) Suppose, on the contrary, that $Y,-Y\in\CONE(\mathcal G)$ for some nonzero $Y$.
Let $(\pi,\rho^\pi,\mathscr H)$ be a unitary representation of $\mathcal G$. 
For every $v\in\mathscr H^\infty$,
\[
0\leq\langle i\rho^\pi(Y)v,v\rangle\leq 0
\]
which implies that $\langle i\rho^\pi(Y)v,v\rangle=0$. Therefore  
for every $v,w\in\mathscr H^\infty$ and every $z\in\mathbb C$,
\begin{equation*}
\begin{split}
0&=\langle \big(i\rho^\pi(Y)\big)(v+zw),v+zw\rangle\\
&=\langle i\rho^\pi(Y)v,v\rangle
+\overline z\langle i\rho^\pi(Y)v,w\rangle
+z\langle i\rho^\pi(Y)w,v\rangle
+|z|^2\langle i\rho^\pi(Y)w,w\rangle \\
&=\overline z\langle i\rho^\pi(Y)v,w\rangle+
z\langle i\rho^\pi(Y)w,v\rangle
\end{split}
\end{equation*}
and since $z$ is arbitrary, $\langle i\rho^\pi(Y)v,w\rangle=0$ 
for every $v,w\in\mathscr H^\infty$. This means that 
$\rho^\pi(Y)=0$, hence $Y  = 0$ because $\cG$ is $\star$-reduced. 

(ii) That $\mathsf\Omega_\lambda$ is positive semidefinite is immediate from the definition of 
$\CONE(\mathcal G)^\star$.
If $X\in\germ g_1$ satisfies $\mathsf\Omega_\lambda(X,X)=0$ then 
from $\lambda\in\INT(\CONE(\mathcal G)^\star)$ it follows that $[X,X]=0$.
Since $\mathcal G$ is \STAR -reduced, Proposition \ref{reducedlemma} implies that $X=0$.

(iii) Part (i) implies that $\CONE(\mathcal G)$ is pointed,
and therefore $\INT(\CONE(\mathcal G)^\star)$ is nonempty \cite[Prop. V.1.5]{Ne00}.
The action of the compact group $\mathrm{INN}_{\germ g_\eev}(\germ k_\eev)$ on $\CONE(\mathcal G)^\star$
leaves
$\INT(\CONE(\mathcal G)^\star)$ invariant. 
By Lemma~\ref{lemma-haarmeasure}, this action
has a fixed point $\lambda\in\INT(\CONE(\mathcal G)^\star)$. Therefore the symmetric 
bilinear form $\mathsf\Omega_\lambda$ of Part (ii) 
is positive definite and invariant with respect to 
$\mathrm{INN}_{\germ g}(\germ k_\eev)$.
From 
the inclusion $\mathrm{Aut}(\germ g)\subseteq\mathrm{Aut}(\germ g_\eev)\times\mathrm{GL}(\germ g_\ood)$
it follows that $\mathrm{INN}_\germ g(\germ k_\eev)$ is compact.

(iv) By Part (iii) it is enough to prove the existence of 
a Cartan subalgebra
which is compactly embedded in $\germ g_\eev$.
 Part (i) implies that $\CONE(\mathcal G)$ is
pointed.
The equality $\germ g_\eev=[\germ g_\ood,\germ g_\ood]$ means 
that  $\CONE(\mathcal G)$ is generating.  
Therefore Lemma~\ref{lemma-compact-cartan} completes the proof. 

(v) Part (i) implies that $\INT(\CONE(\mathcal G)^\star)\neq\emptyset$.
Since $\mathrm{INN}_{\germ g_\eev}(\germ h_\eev)$ is compact
and leaves 
$\INT(\CONE(\mathcal G)^\star)$
invariant, Lemma \ref{lemma-haarmeasure} implies  
that there exists a $\mu\in\INT(\CONE(\mathcal G)^\star)$ which is fixed by 
$\mathrm{INN}_{\germ g_\eev}(\germ h_\eev)$, i.e., 
contained in $\mathbf p^*(\germ h_\eev^*)$. 
\end{proof}

\begin{proposition}
\label{prop-computable-criterion}
Suppose that $\mathcal G=(G,\germ g)$ is a \STAR -reduced Lie supergroup.
Let 
\begin{description}[iiiiii]
\item[\rm(i)] $\germ h_\eev$ be a Cartan subalgebra of 
$\germ g_\eev$ which is compactly embedded in $\germ g$,
\item[\rm(ii)] $\Delta$ be the root system associated to $\germ h_\eev$
(see Proposition \ref{propsupercompactcartan}),   
\item[\rm(iii)] $\mu\in\INT(\CONE(\mathcal G)^\star)\cap\mathbf p^*(\germ h_\eev^*)$, where
$\mathbf p^*$ is the map defined in the statement of Proposition \ref{propertiesofstarreduced}.
\end{description}
Then for every nonzero $\alpha\in\Delta$ the Hermitian form 
\[
\langle\cdot,\cdot\rangle_\alpha:\germ g_\ood^{\mathbb C,\alpha}
\times
\germ g_\ood^{\mathbb C,\alpha}\to \mathbb C
\]
defined by $\langle X,Y\rangle_\alpha=\mu([X,\overline Y])$
is positive definite.
\end{proposition}
\begin{proof}
Let 
\[
\mathsf\Omega_\mu:\germ g_\ood\times\germ g_\ood\to\mathbb R
\] 
be the symmetric bilinear form defined by 
\[
\mathsf\Omega_\mu(X,Y)=\mu([X,Y]).
\]
By Proposition \ref{propertiesofstarreduced}(ii) the form $\mathsf\Omega_\mu$ is positive definite. 
If $X\in\germ g_\ood^{\mathbb C,\alpha}$ then $\overline X\in\germ g_\ood^{\mathbb C,-\alpha}$
and 
\begin{equation*}
\begin{split}
\mathsf\Omega_\mu(X+\overline X,X+\overline X)&=
\mu([X+\overline X,X+\overline X])\\
&=\mu([X,X])+\mu([X,\overline X])+\mu([\overline X,X])+
\mu([\overline X,\overline X]).
\end{split}
\end{equation*}
But $[X,X]\in\germ g_\eev^{\mathbb C,2\alpha}$ and  
$[\overline X,\overline X]\in\germ g_\eev^{\mathbb C,-2\alpha}$, and from
$\mu\in\mathbf p^*(\germ h_\eev^*)$ and $\alpha\neq 0$
it follows that 
\[
\mu([X,X])=\mu([\overline X,\overline X])=0.
\] 
Consequently
\[
\mu([X,\overline X])=\frac{1}{2}\mathsf\Omega_\mu(X+\overline X,X+\overline X)\geq 0,
\]
and $\mu(X,\overline X])=0$ implies that $X=0$.

Moreover, if $\mu([X,\overline{X}])=0$ then $\mathsf\Omega_\mu(X+\overline X,X+\overline X)$
from which it follows that $X+\overline{X}=0$. This means that $i X\in\germ g$, hence 
$[\germ h_\eev,iX]\subseteq\germ g$.
However, if $H\in\germ h_\eev$ is chosen such that $\alpha(H)\neq 0$, then
\[
[H,iX]=i[H,X]=i^2\alpha(H)X=-\alpha(H)X
\]
and this yields a contradiction because clearly $-\alpha(H)X\notin\germ g$.

\end{proof}

\subsection{Application to real simple Lie superalgebras}
Let $\mathcal G=(G,\germ g)$ be a Lie supergroup such that $G$ is connected and 
$\germ g$ is a real simple Lie superalgebra with nontrivial odd part.
Assume that $\mathcal G$ has 
nontrivial unitary representations. The goal of this section is use
the necessary conditions obtained in Section \ref{section-properties}
to obtain strong conditions on $\germ g$.

Since $\germ g$ is simple, $\mathcal G$ 
will be \STAR -reduced and
Proposition~\ref{propertiesofstarreduced}(iv)
 implies that
$\germ g_\eev$ contains a compactly embedded Cartan subalgebra. In particular,
since complex simple Lie algebras do not have compactly embedded Cartan subalgebras, 
$\germ g$ should be a real form of a 
complex simple Lie superalgebra. However, as Theorem \ref{prop-simplecase}
below shows, for a large class of these real forms there are no nontrivial
unitary representations. For simplicity, we exclude the real forms of 
exceptional cases
$\mathbf G(3)$, $\mathbf F(4)$ and $\mathbf D(2|1,\alpha)$.

\begin{theorem}
\label{prop-simplecase}
If $\germ g$ is one of the following Lie superalgebras then 
$\mathcal G$ does not have any nontrivial unitary representations.
\begin{description}[iiiiiii]
\item[\rm(i)] $\germ{sl}(m|n,\mathbb R)$ where $m>2$ or $n>2$.
\item[\rm(ii)] $\germ{su}(p,q|r,s)$ where $p,q,r,s>0$.
\item[\rm(iii)] $\germ{su}^*(2p,2q)$ where $p,q>0$ and $p+q>2$.
\item[\rm(iv)] $\germ{p\overline{q}}(m)$ where $m>1$.
\item[\rm(v)] $\germ{usp}(m)$ where $m>1$.
\item[\rm(vi)] $\germ{osp}^*(m|p,q)$ where $p,q,m>0$.
\item[\rm(vii)] $\germ{osp}(p,q|2n)$ where $p,q,n>0$.
\item[\rm(viii)] Real forms of $\mathbf P(n)$, $n>1$.
\item[\rm(ix)] $\germ{psq}(n,\mathbb R)$ where $n>2$, $\germ{psq}^*(n)$ where $n>2$, and \\
$\germ{psq}(p,q)$, where $p,q>0$.
\item[\rm(x)] Real forms of $\mathbf W(n)$, $\mathbf S(n)$, and $\tilde{\mathbf S}(n)$.
\item[\rm(xi)] $\mathbf H(p,q)$ where $p+q>4$.
\end{description}
\end{theorem}

\begin{proof} 
Throught the proof, for every $n$ we denote the $n\times n$ identity matrix
by  $\mathrm I_n$, and 
set
\[
\mathrm{I}_{p,q}=\begin{bmatrix}
\mathrm I_p&0\\
0&-\mathrm I_q
\end{bmatrix}
\quad\textrm{and}\quad
\mathrm{J}_n=
\begin{bmatrix}
0&\mathrm I_n\\
-\mathrm I_n&0
\end{bmatrix}.
\]
(i) Since $[\g_{\ood}, \g_\ood] \cong \fsl(m,\R) \oplus \fsl(n,\R)$ 
has no compactly embedded Cartan subalgebra, this follows from 
Proposition \ref{propertiesofstarreduced}(iv).

(ii) In the standard realization  of $\germ{sl}(p+q|r+s,\mathbb C)$ as 
quadratic matrices of size $p+q + r+s$, 
$\germ{su}(p,q|r,s)$
can be described as
\[
\left\{
\begin{bmatrix}
A&B\\
C&D
\end{bmatrix}\in \fsl(p+q|r+s,\C) \bigg|\ 
\begin{bmatrix}
-\II_{p,q}A^*\II_{p,q}\ \ &i\II_{p,q}C^*\II_{r,s}\\
i\II_{r,s}B^*\II_{p,q}\ \ &-\II_{r,s}D^*\II_{r,s}
\end{bmatrix}
=
\begin{bmatrix}
A&B\\
C&D
\end{bmatrix}
\right\}.
\]
Suppose, on the contrary, that $\mathcal G$ is \STAR -reduced.
Proposition \ref{propertiesofstarreduced}(iii) implies that 
the diagonal matrices in $\germ{su}(p,q|r,s)$
constitute a Cartan subalgebra of $\germ{su}(p,q|r,s)_\eev$
which is compactly embedded in $\germ{su}(p,q|r,s)$. Let $\mu$ be chosen as in Propostion
\ref{prop-computable-criterion}. 
For every $a\leq r$ and $b\leq p$, the 
matrix
\[
X_{a,b}=
\begin{bmatrix}
0&0\\
\mathrm{E}_{a,b}&0
\end{bmatrix}
\]
is a root vector.
Let $\tau$ denote 
the complex conjugation corresponding to the above realization of $\germ{su}(p,q|r,s)$. 
One can easily check that
\[
\tau(X_{a,b})=
\begin{bmatrix}
0&i\mathrm{E}_{b,a}\\
0&0
\end{bmatrix}.
\]
Set $H_{a,b}=[X_{a,b},\tau(X_{a,b})]$. It is easily checked that
\[
H_{a,b}=
\begin{bmatrix}
i\mathrm{E}_{b,b}&0\\
0&i\mathrm{E}_{a,a}
\end{bmatrix}.
\]
For $a$ and $b$ there are three other possibilities to consider. 
If $a\leq r$ and $b>p$, or if $a>r$ and $b\leq p$, then
\[
H_{a,b}=
\begin{bmatrix}
-i\mathrm{E}_{b,b}&0\\
0&-i\mathrm{E}_{a,a}
\end{bmatrix},
\]
and if $a>r$ and $b>p$ then
\[
H_{a,b}=
\begin{bmatrix}
i\mathrm{E}_{b,b}&0\\
0&i\mathrm{E}_{a,a}
\end{bmatrix}.
\]
Proposition \ref{prop-computable-criterion} implies that
$\mu(H_{a,b})>0$ for every $1\leq a\leq p+q$ and every $1\leq b\leq r+s$.
However, from the assumption that $p,q,r,$ and $s$ are all positive,
it follows that
the zero matrix 
lies in the convex hull of the $H_{a,b}$'s, which is a contradiction. 
Therefore $\mathcal G$ cannot be \STAR -reduced.

(iii) Note that 
$\germ{su}^*(2p|2q)_\eev\simeq\germ{su}^*(2p)\oplus\germ{su}^*(2q)$.
The maximal compact subalgebra of $\germ{su}^*(2n)$
is $\germ{sp}(n)$, which has rank $n$. The rank of the complexification
of $\germ{su}^*(2n)$, which is $\germ{sl}(2n,\mathbb C)$, is $2n-1$. 
If $n>1$, then $2n-1>n$ implies that 
$\germ{su}^*(2n)$ does not have a compactly embedded Cartan subalgebra.
Now use Proposition \ref{propertiesofstarreduced}(i) and 
Lemma \ref{lemma-compact-cartan}.

(iv) This Lie superalgebra is a quotient of $\overline{\germ{q}}(m)$ by its center, 
where $\overline{\germ{q}}(m)$ is defined in the standard
realization of $\germ{sl}(m|m,\mathbb C)$ by
\[
\overline{\germ q}(m)=\left\{
\begin{bmatrix}
A&B\\
C&D
\end{bmatrix}  \in \fsl(m|m,\C)
\ \bigg|\ 
\begin{bmatrix}
\overline{D}&\overline{C}\\
\overline{B}&\overline{A}
\end{bmatrix}
=
\begin{bmatrix}
A&B\\
C&D
\end{bmatrix}
\right\}.
\]
One can now use Proposition \ref{propertiesofstarreduced}(iv) 
because $\oline\fq(m)_\eev \cong \fsl(m,\C) \oplus \R$ contains no 
compactly embedded Cartan subalgebra.

(v) This Lie superalgebra is a quotient of $\germ{up}(m)$ by its center, 
where $\germ{up}(m)$ is defined in the standard
realization of $\germ{sl}(m|m,\mathbb C)$ by
\[
\germ{up}(m)=
\left\{
\begin{bmatrix}
A&B\\
C&D
\end{bmatrix} \in \fsl(m|m,\C) 
\ \bigg|\ 
\begin{bmatrix}
-D^*&B^*\\
-C^*&-A^*
\end{bmatrix}
=
\begin{bmatrix}
A&B\\
C&D
\end{bmatrix}
\right\}.
\]
This implies that $\fup(m)_\eev \cong \mathrm{sl}(m,\C)\oplus\mathbb{R}$. 
Since this Lie algebra has no compactly embedded Cartan subalgebra, 
the assertion follows from 
Proposition \ref{propertiesofstarreduced}(iv).

(vi) From Section \ref{section-characterization} it follows that
$\germ{osp}^*(m|p,q)_\eev\simeq\germ{so}^*(m)\oplus\germ{sp}(p,q)$
has pointed generating invariant cones if and only if $p=0$ or $q=0$.
One can now use Proposition \ref{propertiesofstarreduced}(i).

(vii) The argument for this case is quite 
similar to the one given for $\germ{su}(p,q|r,s)$, i.e.,
the idea is to find root vectors $X_\alpha\in\germ g_\ood^{\mathbb C,\alpha}$ such that 
the convex hull of the $[X_\alpha,\tau(X_\alpha)]$'s contains the origin.
The details are left to the reader, but
it may be helpful to illustrate how one can find the root vectors.
The complex simple Lie superalgebra $\germ{osp}(m|2n,\mathbb C)$ can be realized inside 
$\germ{sl}(m|2n,\mathbb C)$ as 
\[
\germ{osp}(m|2n,\mathbb C)=
\left\{
\begin{bmatrix}
A&B\\
C&D
\end{bmatrix}
\ \bigg|\ 
\begin{bmatrix}
-A^t\ &-C^t\JJ_n\\
-\JJ_nB^t\ &\JJ_nD^t\JJ_n
\end{bmatrix}
=
\begin{bmatrix}
A&B\\
C&D
\end{bmatrix}
\right\}.
\]
If $p$ and $q$ are nonnegative integers
satisfying $p+q=m$ then 
$\germ{osp}(p,q|2n)$ is the set of fixed points of the map
\[
\tau:\germ{osp}(m|2n,\mathbb C)\to\germ{osp}(m|2n,\mathbb C)
\]
defined by
\[
\tau\left(
\begin{bmatrix}
A&B\\
C&D
\end{bmatrix}
\right)
=
\begin{bmatrix}
\II_{p,q}\overline{A}\II_{p,q}\ &\II_{p,q}\overline{B}\\[3pt]
\overline{C}\II_{p,q}\ &\overline{D}
\end{bmatrix}.
\]
Moreover, $\germ{osp}(p,q|2n)_\eev\simeq\germ{so}(p,q)\oplus\germ{sp}(2n,\mathbb R)$
consists of block diagonal matrices, i.e., matrices for which $B$ and $C$ are zero.

Assume that $\germ{osp}(p,q|2n)$ is \STAR -reduced.
Then the span of 
\[
\Big\{\mathrm{E}_{j,p+1-j}-\mathrm{E}_{p+1-j,j}\ 
|\ 1\leq j\leq \lfloor\frac{p}{2}\rfloor\Big\}
\]
and 
\[
\Big\{
\mathrm{E}_{p+j,p+q+1-j}-\mathrm{E}_{p+q+1-j,p+j}\ 
|\ 1\leq j\leq \lfloor\frac{q}{2}\rfloor
\Big\}
\]
is a compactly embedded Cartan subalgebra of  
$\germ{so}(p,q)$, and the span of 
\[
\{\mathrm{E}_{p+q+j,p+q+n+j}-\mathrm{E}_{p+q+n+j,p+q+j}\ |\ 1\leq j\leq n\}
\]
is a compactly embedded Cartan subalgebra of $\germ{sp}(2n,\mathbb R)$.

Fix $1\leq b\leq n$.
For every
$a\leq p$ we can obtain two root vectors as follows. If we
set
\[
B_{a,b}=\mathrm{E}_{a,b}+i\mathrm{E}_{a,b+n}+i\mathrm{E}_{p+1-a,b}-\mathrm{E}_{p+1-a,b+n}
\]
and 
\[
C_{a,b}=-i\mathrm{E}_{b,a}+\mathrm{E}_{b,p+1-a}+\mathrm{E}_{b+n,a}+i\mathrm{E}_{b+n,p+1-a},
\]
then the matrix
\[
X_{a,b}=
\begin{bmatrix}
0&B_{a,b}\\
C_{a,b}&0
\end{bmatrix}
\]
is a root vector, and
$H_{a,b}=[X_{a,b},\tau(X_{a,b})]$ is given by
\[
H_{a,b}=
\begin{bmatrix}
A_{a,b}&0\\
0&D_{a,b}
\end{bmatrix}
\]
where $A_{a,b}=2\mathrm{E}_{a,p+1-a}-2\mathrm{E}_{p+1-a,a}$
and $D_{a,b}=-2\mathrm{E}_{b,b+n}+2\mathrm{E}_{b+n,b}$.
Similarly, setting
setting 
\[
B_{a,b}=\mathrm{E}_{a,b}-i\mathrm{E}_{a,b+n}+i\mathrm{E}_{p+1-a,b}+\mathrm{E}_{p+1-a,b+n}
\]
and 
\[
C_{a,b}=i\mathrm{E}_{b,a}-\mathrm{E}_{b,p+1-a}+\mathrm{E}_{b+n,a}+i\mathrm{E}_{b+n,p+1-a}
\]
yields another root vector $X_{a,b}$, and 
in this case for the corresponding $H_{a,b}$ we have
\[
A_{a,b}=-2\mathrm{E}_{a,p+1-a}+2\mathrm{E}_{p+1-a,a}
\]
and 
\[
D_{a,b}=-2\mathrm{E}_{b,b+n}+2\mathrm{E}_{b+n,b}.
\]

Moreover, 
when $p$ is odd, 
setting 
\[
B_{\lceil \frac{p+1}{2}\rceil,b}
=
\mathrm{E}_{\lceil\frac{p+1}{2}\rceil,b}
+
i\mathrm{E}_{\lceil\frac{p+1}{2}\rceil,b+n}
\]
and 
\[
C_{\lceil \frac{p+1}{2}\rceil,b}
=
-i\mathrm{E}_{b,\lceil\frac{p+1}{2}\rceil}+\mathrm{E}_{b+n,\lceil\frac{p+1}{n}\rceil}
\]
yields a root vector $X_{\lceil \frac{p+1}{2}\rceil,b}$,
and 
$H_{\lceil\frac{p+1}{2}\rceil,b}$ is given by
\[
A_{\lceil\frac{p+1}{2}\rceil,b}=0\]
and 
\[
D_{\lceil\frac{p+1}{2}\rceil,b}=
-2\mathrm{E}_{b,b+n}+2\mathrm{E}_{b+n,b}.
\]
The case $p<a\leq p+q$  is similar.

(viii) Follows from Proposition \ref{prop-symmetryofrootsystem}, as the root system of
$\mathbf P(n)$ is not 
symmetric.

(ix) For $\germ{psq}(n,\mathbb R)$ and $\germ{psq}^*(n)$, use  
Proposition \ref{propertiesofstarreduced}(iv)
and the fact that 
\[
\germ{psq}(n,\mathbb R)_\eev\simeq\germ{sl}(n,\mathbb R)
\quad\textrm{and}\quad
\germ{psq}^*(n)_\eev\simeq\germ{su}^*(n)
.
\]

For $\germ{psq}(p,q)$ and $p,q > 0$, 
we observe that it is a quotient of the subsuperalgebra 
$\tilde{\g}$ of $\germ{sl}(p+q|p+q,\mathbb C)$ given by 
\[ \tilde{\g} = \fsq(p,q) = 
\left\{
\begin{bmatrix}
A&B\\
B&A
\end{bmatrix} 
\ \bigg|\ 
\begin{bmatrix}
-\II_{p,q}A^*\II_{p,q}\ &i\II_{p,q}B^*\II_{p,q}\\
i\II_{p,q}B^*\II_{p,q}\ &-\II_{p,q}A^*\II_{p,q}
\end{bmatrix}
=
\begin{bmatrix}
A&B\\
B&A
\end{bmatrix}
\right\}.
\]
Let $\zeta \in \C$ be a squareroot of $i$.
Then the maps
\[
\fu(p,q) \to \tilde{\g_\eev}\quad,\quad
 A \mapsto \pmat{A & 0 \\ 0 & A}
\]
and 
\[
\fu(p,q) \to \tilde{\g_\ood}\quad,\quad
B \mapsto \pmat{0 & \zeta^{-1} B\\ \zeta^{-1} B & 0}
 \] 
are linear isomorphisms. Note that 
$\fk_\eev = \fu(p) \oplus \fu(q)$ is a maximal compactly embedded 
subalgebra of $\tilde{\g}_\eev$. Its center is 
\[ \mathscr Z(\fk_\eev) = \R i \mathrm{I}_p \oplus \R i \mathrm{I}_q \] 
and $\tilde{\g}_\eev$ is quasihermitian. The projection 
$p_\z \: \fu(p,q) \to \mathscr Z(\fk_\eev)$ is simply given by 
\[ p_\z\pmat{a & b \\ b^* & d} = 
\pmat{\frac{1}{p} \tr(a)\mathrm{I}_p & 0 \\ 0 & \frac{1}{q} \tr(d)\mathrm{I}_q}. \] 

Let $C \subeq \tilde{\g}_\eev$ be the closed convex cone generated by 
$[X,X]$, $X \in \tilde{\g}_\ood$. Since $\tilde{\g}_\eev$ is quasihermitian, 
Lemma~\ref{lem:qherm} implies that 
$p_\z(C) = C \cap \mathscr Z(\fk_\eev)$. 

Next we observe that 
\[ \pmat{0 & \zeta^{-1} B\\ \zeta^{-1} B & 0}^2 
= \pmat{-i B^2 & 0 \\ 0 & -i B^2} \quad \mbox{ for every} \quad B \in \fu(p,q).\] 
For 
$B = \pmat{a & b\\ b^* & d}$ we have 
\[ B^2 
= \pmat{a & b\\ b^* & d}^2 
= \pmat{a^2 + bb^* & ab + bd \\ b^*a + a b^* & b^*b + d^2}, \] 
so that 
\[ p_\z(-iB^2) 
= -i \pmat{\frac{1}{p}\tr(a^2 + bb^*) & 0 \\ 0 & 
\frac{1}{q}\tr(b^*b + d^2)}.\] 
Applying this to positive multiples of matrices where 
only the $a$, $b$ or $d$-component is non-zero, we see that the closed 
convex cone $p_\z(C)$ contains the elements 
\[ Z_1 = \pmat{i\mathrm I_p & 0 \\ 0 & 0}, \quad 
Z_2 = \pmat{0 & 0 \\ 0 & i\mathrm I_q} \quad \mbox{ and } \quad 
 Z_3 = - \pmat{\frac{1}{p}i\mathrm{I}_p & 0 \\ 0 & \frac{1}{q}i\mathrm I_q}.\] 
This implies that $p_\z(C) = \mathscr Z(\fk_\eev) \subeq C$. 

We conclude that $\mathscr Z(\tilde{\g}_\eev) = i \R\mathrm I_{p+q}\subeq C$, so that 
$C = \mathscr (\tilde{\g}_\eev) \oplus C_1$, where 
$C_1 = C \cap [\tilde{\g}_\eev,\tilde{\g}_\eev]$ 
is a non-pointed non-zero invariant closed convex 
cone in a simple Lie algebra isomorphic to $\fsu(p,q)$. This leads to 
$C_1 =[\tilde{\g}_\eev,\tilde{\g}_\eev]$. We conclude that 
$C = \tilde{\g}_\eev$ and the same holds also for the quotient 
$\fpsq(p,q)$.

(x) Follows from Proposition \ref{prop-symmetryofrootsystem}, as 
the root systems of
these complex simple Lie superalgebras are not 
symmetric (see \cite[App. A]{penkovstrongly}).

(xi) Suppose, on the contrary, that 
$\mathcal G$ is \STAR -reduced.
Proposition \ref{propertiesofstarreduced}(i)
and
Lemma \ref{lemma-abelianideal} imply that every abelian ideal of 
$\germ g =  \mathbf H(p,q)$ lies in its center.
The standard $\mathbb Z$-grading of $\mathbf H(p+q)$ (see \cite[Prop. 3.3.6]{kac})
yields a grading of $\mathbf H(p+q)_\eev$, i.e.,
\[
\mathbf H(p+q)_\eev=\mathbf H(p+q)_\eev^{(0)}\oplus
\mathbf H(p+q)_\eev^{(2)}\oplus
\cdots
\oplus\mathbf H(p+q)_\eev^{(k)}
\]
where $k=p+q-3$ if $p+q$ is odd and $k=p+q-4$ otherwise. This grading is consistent with the
real form $\mathbf H(p,q)_\eev$. Since 
$\mathbf H(p,q)_\eev^{(k)}$ is an abelian ideal of 
$\mathbf H(p,q)_\eev$, it should lie in the center
of $\mathbf H(p,q)_\eev$. It follows that $\mathbf H(p+q)_\eev^{(k)}$ lies in the center
of $\mathbf H(p+q)_\eev$. However, this is impossible because
it is known (see \cite[Prop. 3.3.6]{kac}) that 
$\mathbf H(p+q)_\eev^{(0)}\simeq\germ{so}(p+q,\mathbb C)$ and
the 
representation of $\mathbf H(p+q)_\eev^{(0)}$ on $\mathbf H(p+q)_\eev^{(k)}$ is isomorphic to
$\wedge^{k+2}\mathbb C^{p+q}$, from which it follows that 
\[
[\mathbf H(p+q)_\eev^{(0)},\mathbf H(p+q)_\eev^{(k)}]\neq\{0\}.
\qedhere \]
\end{proof}

\begin{remark}
In classical cases, Theorem \ref{prop-simplecase} can be viewed as a converse
to the classification of 
highest weight modules obtained in \cite{jakobsen}.  
From Theorem \ref{prop-simplecase} it also follows that 
for the nonclassical cases, unitary representations are rare.
\end{remark}
\begin{remark}
The results of \cite{jakobsen} imply that 
real forms of $\mathbf A(m|m)$ 
do not have any unitarizable highest weight modules.
However, $\mathbf A(m|m)$ is a 
quotient of $\germ{sl}(m|m,\mathbb C)$, and there exist unitarizable modules of 
$\germ{su}(p,m-p|m,0)$ which do not factor to the simple quotient. For instance, the standard representation is
a finite dimensional unitarizable module of $\germ{su}(m,0|m,0)$ with this property.
\end{remark}

\subsection{Application to real semisimple Lie superalgebras}

Although real semisimple Lie superalgebras may have a complicated structure, 
those 
which have faithful unitary representations are relatively easy to describe.

Given a finite dimensional real Lie superalgebra $\germ g$, 
let us call it 
\STAR -reduced if there exists a \STAR -reduced 
Lie supergroup $\mathcal G=(G,\germ g)$.

\begin{theorem}
Let $\mathcal G=(G,\germ g)$ be a \STAR -reduced Lie supergroup.
If $\germ g$ is a real semisimple Lie superalgebra then there exist
\STAR -reduced real simple Lie superalgebras $\germ s_1,\ldots,\germ s_k$ such that
\[
\germ s_1\oplus\cdots\oplus\germ s_k\subseteq \germ g\subseteq
\mathrm{Der}_{\mathbb R}(\germ s_1)\oplus\cdots\oplus\mathrm{Der}_{\mathbb R}(\germ s_k).
\]
\end{theorem}

\begin{proof}
We use the description of $\germ g$ given in 
Theorem~\ref{classification-of-semisimple}. First note that for every $i$ we have $n_i=0$. 
To see this, suppose on the contrary
that $n_i>0$ for some $i$, and let $\xi_1,\ldots,\xi_{n_i}$ be the standard generators of 
$\mathbf \Lambda_{\mathbb K_i}(n_i)$.
For every nonzero $X\in({\germ s_i})_\eev$ have $X\otimes \xi_1\in({\germ s_i})_\ood$ and 
\[
[X\otimes\xi_1,X\otimes \xi_1]=0.
\] 
Proposition \ref{reducedlemma} implies that
$X\otimes \xi_1$ lies in the kernel of every unitary representation of $\mathcal G$, which is a contradiction.

From the fact that all of the $n_i$, $1\leq i\leq k$,  are zero it follows that
\[
\germ s_1\oplus\cdots\oplus\germ s_k
\subseteq\germ g\subseteq
\mathrm{Der}_{\mathbb K_1}(\germ s_1)\oplus\cdots\oplus\mathrm{Der}_{\mathbb K_k}(\germ s_k)
\]
and from $\germ s_i\subseteq\germ g$ it follows that every $\germ s_i$ is 
\STAR -reduced.
\qedhere\end{proof}

\subsection{Application to nilpotent Lie supergroups}

Another interesting by-product of the results of Section 
\ref{section-properties} 
is the following statement about unitary 
representations of
nilpotent Lie supergroups.
(A Lie supergroup $\mathcal G=(G,\germ g)$ is called \emph{nilpotent} if $\germ g$ is nilpotent.)

\begin{theorem}
If $(\pi,\rho^\pi,\mathscr H)$ is a unitary representation of a nilpotent Lie supergroup $(G,\germ g)$
then $\rho^\pi([\germ g_\ood,[\germ g_\ood,\germ g_\ood]])=\{0\}$.
\end{theorem}
\begin{proof}
By passing to a quotient one can see that it suffices to show that if $(G,\germ g)$ is nilpotent and \STAR -reduced then
$[\germ g_\ood,[\germ g_\ood,\germ g_\ood]]=\{0\}$.
Without loss of generality one can assume that $\germ g_\eev=[\germ g_\ood,\germ g_\ood]$.
By 
Proposition \ref{propertiesofstarreduced}(iv) there exists a 
Cartan subalgebra $\germ h_\eev$ of $\germ g_\eev$ which is compactly
embedded in $\germ g$. As $\g_\eev$ is nilpotent, we have 
$\g_\eev = \h_\eev$. 
Proposition \ref{propsupercompactcartan}  
implies that $\germ g_\eev$ acts semisimply on $\germ g$.
Nevertheless, since $\germ g$ is nilpotent, for every
 $X\in\germ g_\eev$ the linear map 
\[
\mathrm{ad}(X):\germ g\to\germ g
\] 
is nilpotent. It follows that
$[\germ g_\eev,\germ g]=\{0\}$. 
In particular, 
$[\germ g_\ood,[\germ g_\ood,\germ g_\ood]]=\{0\}$.
\end{proof}

\section{Highest weight theory}

For Lie supergroups whose Lie algebra $\g$ is generated by 
its odd part, we analyse in this section the structure 
of the irreducible unitary representations. 
The main result is Theorem~\ref{thm:hiweistr} which asserts 
that this structure is quite similar to the structure of 
highest weight modules. Here it is generated by an irreducible 
representation of a Clifford Lie superalgebra and not simply by a 
an eigenvector. 


\subsection{A Fr\'echet space of analytic vectors} 

Let $G$ be a connected Lie group with Lie algebra 
$\germ g$. Let $\germ t \subseteq \germ g$ be a compactly embedded Cartan subalgebra, 
and $T =\exp(\germ t)$ be the corresponding subgroup of $G$. 
Then $\germ g^\mathbb C$ carries a norm $\|\cdot\|$ which is invariant under 
$\mathrm{Ad}(T)$. 
In particular, for each $r > 0$, the open ball 
$B_r = \{ X \in \germ g^\mathbb C \: \|X\| < r \}$ is an open subset 
which is invariant under $\mathrm{Ad}(T)$. 

Let $(\pi, \mathscr H)$ be a unitary representation of $G$. 
A smooth vector $v \in \mathscr H^\infty$ is {\it analytic} if and only if 
there exists an $r > 0$ such that the power series 
\begin{equation}\label{eq:posi} 
f_v : B_r \to \mathscr H, \quad 
f_v(X) = \sum_{n = 0}^\infty \frac{1}{n!} \mathsf{d}\pi(X)^nv  
\end{equation}
defines a holomorphic function on $B_r$. 
In fact, if the series \eqref{eq:posi} converges on some $B_r$, 
then it defines a 
holomorphic function, and  
the theory of analytic vectors for unitary one-parameter groups 
implies that $f_v(X) = \pi(\exp(X))v$
for every $X \in B_r \cap \germ g$. Therefore the orbit 
map of $v$ is analytic. 

If the series \eqref{eq:posi} converges on $B_r$, it converges 
uniformly on $B_s$ for every $s < r$ (\cite[Prop.~4.1]{BS71}). 
This means that 
the seminorms 
\[ 
q_n(v) = \sup\{ \|\mathsf{d}\pi(X)^nv\| \ \Big| \ \|X\| \leq 1, X \in \germ g\}
\] 
satisfy 
\[
\sum_{n = 0}^\infty \frac{s^n}{n!} q_n(v) < \infty\textrm{ for every }s<r.
\]
Note that the seminorms $q_n$ define the topology of $\mathscr H^\infty$ 
(cf.\ \cite[Prop.~4.6]{Ne10}). 

For every $r>0$, let $\mathscr H^{\omega,r}$ denote the set of all analytic vectors 
for which \eqref{eq:posi} converges on $B_r$, so that 
\[ \mathscr H^\omega = \bigcup_{r > 0} \mathscr H^{\omega,r}.\] 
If $v \in \mathscr H^{\omega,r}$ and $s < r$, set 
\[ p_s(v) = \sum_{n = 0}^\infty \frac{s^n}{n!} q_n(v) \] 
and note that this is a norm on  $\mathscr H^{\omega,r}$. 

\begin{lemma} \label{lem:2.1x} The norms $p_s$, $s < r$, turn 
$\mathscr H^{\omega,r}$ into a Fr\'echet space. 
\end{lemma}

\begin{proof} Since $p_s < p_t$ for $s < t < r$, the topology 
on $\mathscr H^{\omega,r}$ is defined by the sequence of seminorms 
$(p_{s_n})_{n \in \mathbb{N}}$ for any sequence $(s_n)$ with $s_n \to r$. 
Therefore $\mathscr H^{\omega,r}$ is metrizable and we have to show that 
it is complete. 

If $(v_n)$ is a Cauchy sequence in $\mathscr H^{\omega,r}$ then 
for every $s < r$ the sequence 
$f_{v_n} : B_r \to \mathscr H$ of 
holomorphic functions converges uniformly 
on each $B_s$ to some function $f : B_r \to \mathscr H$, 
which implies that $f$ is holomorphic. 

Let $v = f(0)$. Then, for each $X \in \germ g$ and $k \in \mathbb{N}$, 
$\mathsf{d}\pi(X)^k v_n$ is a Cauchy sequence in $\mathscr H$. This implies that 
$v \in \mathscr H^\infty$ with 
$\mathsf{d}\pi(X)^k v_n \to \mathsf{d}\pi(X)^kv$ for every $X \in \germ g$ and $k \in \mathbb{N}$ 
(\cite[Prop.~3.1]{BS71}).
Therefore $f = f_v$ on $B_r$, and this means that 
$v \in \mathscr H^{\omega,r}$ with $v_n \to v$ in the topology of 
$\mathscr H^{\omega,r}$.
\end{proof}

\begin{lemma}
\label{lem:2.2} 
If $K \subseteq G$ is a subgroup leaving the norm 
$\|\cdot\|$ on $\germ g^\mathbb C$ invariant, then the norms 
$p_s$, $s < r$, on $\mathscr H^{\omega,r}$ are $K$-invariant 
and the action of $K$ on $\mathscr H^{\omega,r}$ is continuous. 
In particular, the action of $K$ on $\mathscr H^{\omega,r}$ integrates to a 
representation of the convolution algebra $L^1(K)$ on $\mathscr H^{\omega,r}$. 
\end{lemma}

\begin{proof} Since $K$ preserves the defining family of norms, 
continuity of the $K$-action on $\mathscr H^{\omega,r}$ follows if we 
show that all orbit maps are continuous 
at $\mathbf 1_K$, where $\mathbf 1_K$ denotes the identity element of $K$. 
Let $v \in \mathscr H^{\omega,r}$ and 
suppose that $k_m \to \mathbf 1_K$ in $K$. 
Then 
\[ p_s(\pi(k_m)v-v) 
= \sum_{n = 0}^\infty \frac{s^n}{n!} q_n(\pi(k_m)v-v) \] 
and \[
q_n(\pi(k_m)v-v) \leq q_n(\pi(k_m)v) + q_n(v) = 2 q_n(v).
\] 
Since $K$ acts continuously on $\mathscr H^\infty$, 
$q_n(\pi(k_m)v-v) \to 0$ for every $n \in \mathbb{N}$, and since 
$p_s(v) < \infty$, the Dominated Convergence Theorem implies that 
$p_s(\pi(k_n)v-v) \to 0$. 

The fact that $\mathscr H^{\omega,r}$ is complete implies that it can be considered 
as a subspace of the product space 
$\prod_{s < r} \mathscr V_s$, where $\mathscr V_s$ denotes the 
completion of $\mathscr H^{\omega,r}$ 
with respect to the norm $p_s$. We thus obtain continuous isometric 
representations of $K$ on the Banach spaces $\mathscr V_s$, which leads by 
integration to representations of $L^1(K)$ on these spaces 
(see \cite[(40.26)]{HR70}). 
Finally, since  $\mathscr H^{\omega,r} \subseteq \prod_{s < r} \mathscr V_s$ is closed by completeness (Lemma~\ref{lem:2.1x}) and 
$K$-invariant, it is also invariant under $L^1(K)$. 
\end{proof}

From now on assume that $r$ is small enough such that the exponential 
function of the simply connected Lie group $\tilde{G}^\mathbb C$ with Lie algebra 
$\germ g^\mathbb C$ maps $B_r$ diffeomorphically onto an open subset of 
$\tilde{G}^\mathbb C$. 
For every $X \in \germ g^\mathbb C$ the corresponding 
left and right invariant vector fields 
define differential operators on $\exp(B_r)$ by 
\[ (L_X f)(g) = \frac{d}{dt}\bigg|_{t=0} f(g \exp(tX))\quad \textrm{ and } \quad 
(R_X f)(g) = \frac{d}{dt}\bigg|_{t=0} f(\exp(tX)g).\] 
Define similar operators $L_X^*$ and $R_X^*$ on $B_r$ by 
\[ L_X^*(f  \circ \exp|_{B_r}) = (L_X f) \circ \exp|_{B_r} 
\quad \textrm{ and } \quad 
R_X^*(f  \circ \exp|_{B_r}) = (R_X f) \circ \exp|_{B_r}.\] 
One can see that
\begin{equation}
  \label{eq:diffop}
L_X^*f_v = f_{\mathsf{d}\pi(X)v} 
\quad \textrm{ and } \quad R_X^*f_v = \mathsf{d}\pi(X) \circ f_v.
\end{equation}
If
$\mathpzc{Hol}(B_r,\mathscr H)$ denotes the Fr\'echet 
space of holomorphic $\mathscr H$-valued functions on 
$B_r$, then the subspace 
$\Hol(B_r, \cH)^\g$ defined by
\[ \Hol(B_r, \cH)^\g = \{ f \in \mathpzc{Hol}(B_r, \mathscr H) \ |\  
R_X^* f = \mathsf{d}\pi(X) \circ f
\textrm{ for every }X\in\germ g\},\] 
is a closed subspace, hence a Fr\'echet space.
Therefore the map
\begin{equation}
  \label{eq:eval}
\mathrm{ev}_0 : \Hol(B_r, \cH)^\g \to \mathscr H, \quad 
f \mapsto f(0) 
\end{equation}
is a continuous linear isomorphism onto $\mathscr H^{\omega,r}$, 
hence a topological isomorphism by the 
Open Mapping Theorem (see \cite[Thm.~2.11]{Ru73}).  

This implies in particular that 

\begin{lemma} \label{lem:invlem} The subspace $\mathscr H^{\omega,r} 
\subeq \cH$ is invariant 
under $\mathscr U(\germ g^\mathbb C)$. 
\end{lemma}


\subsection{Spectral theory for analytic vectors} 

We have already seen in Lemma~\ref{lem:2.2} 
that if $(\pi, \mathscr H)$ is a  
unitary representation of $G$ then the subspaces $\mathscr H^{\omega, r}$ 
are invariant under the action of the convolution algebras of certain 
subgroups $K \subeq G$. As a consequence, we shall now derive 
that elements of 
spectral subspaces of certain unitary one-parameter groups can be approximated 
by analytic vectors. 

We begin by a lemma about the relation between one-parameter groups and 
spectral measures. Let $\germ B(\mathbb R)$ denote the space of 
Borel measurabe functions 
on $\mathbb R$ and $\mathscr S(\mathbb R)$ denote the Schwartz space of 
$\mathbb R$.
\begin{lemma} \label{lem:c.1} 
Let $\gamma : \mathbb R \to \mathbf U(\mathscr H)$ be a unitary representation of
the additive group of $\mathbb R$ and $A = A^* = -i\gamma'(0)$ be its self-adjoint 
generator, so that $\gamma(t) = e^{itA}$ in terms of measurable functional 
calculus. Then the following assertions hold.
\begin{description}
\item[\rm(i)]  For each $f \in L^1(\mathbb R,\mathbb C)$, we have 
$\gamma(f) = \hat f(A),$
where \[
\hat f(x) = \int_\mathbb R e^{ixy} f(y)\, dy
\] is the Fourier transform 
of $f$. 
\item[\rm(ii)]  Let $P : \germ B(\mathbb R) \to \mathcal L(\mathscr H)$ be the unique 
spectral measure with $A = P(\mathrm{id}_\mathbb R)$. 
Then for every closed subset $E \subseteq \mathbb R$ the condition 
$v \in P(E)\mathscr H$ is equivalent to $\gamma(f)v = 0$ for 
every $f \in \mathscr S(\mathbb R)$ with $\hat f\big|_E = 0$. 
\end{description}
\end{lemma}

\begin{proof} 
Since the unitary representation $(\gamma,\mathscr H)$ is a direct sum 
of cyclic representations, it suffices to prove the assertions for 
cyclic representations. Every cyclic representation of 
$\mathbb R$ is equivalent to the representation on 
some space $\mathscr H = L^2(\mathbb R,\mu)$, where $\mu$ is a Borel probability 
measure on $\mathbb R$ and $(\gamma(t)\xi)(x) = e^{itx}\xi(x)$ 
(see \cite[Thm.~VI.1.11]{Ne00}). 

(i) This means that $(A\xi)(x) = x\xi(x)$, so that 
$\hat f(A)\xi(x) = \hat f(x)\xi(x)$. For 
every $f \in L^1(\mathbb R,\mathbb C)$ 
the equalities 
\[ (\gamma(f)\xi)(x) 
= \int_\mathbb R f(t) e^{itx}\xi(x)\, dt = \hat f(x)\xi(x) \] 
hold in the space $\mathscr H = L^2(\mathbb R,\mu)$.

(ii) In terms of functional calculus, we have 
$P(E) = \chi_E(A)$, where $\chi_E$ is the characteristic function of $E$. 
If $\hat f\big|_E = 0$, then Part (i) and the fact that $\hat f \chi_E = 0$ 
imply that 
\[ 0 = (\hat f \cdot \chi_E)(A) 
= \hat f(A) \chi_E(A) = \gamma(f)P(E). 
\]

Conversely, suppose that $v \in \mathscr H$ satisfies 
$\gamma(f)v =0$ for every $f \in \mathscr S(\mathbb R)$ with $\hat f\big|_E = 0$. 
If $v \not\in P(E)\mathscr H$, then $P(E^c)v \neq 0$, and since $E^c$ is open and a 
countable union of compact subsets, there exists a compact subset 
$B \subseteq E^c$ with $P(B)v \neq 0$. 
Let $\psi \in C^\infty_c(\mathbb R)$ be such that $\psi\big|_B = 1$ and 
$\mathrm{supp}(\psi) \subseteq E^c$. Then 
\[ 0\neq P(B) v = \chi_B(A) v 
= (\chi_B \cdot \psi)(A) v = \chi_B(A)\psi(A)v  
\] 
implies that $\psi(A)v \neq 0$. 
Since the Fourier transform defines a bijection 
$\mathscr S(\mathbb R) \to \mathscr S(\mathbb R)$ (\cite{Ru73}), 
there exists an $f \in \mathscr S(\mathbb R)$ 
with $\hat f= \psi$. Then 
$\gamma(f)v = \hat f(A)v = \psi(A)v \neq 0$, contradicting our assumption. 
This implies that $v \in P(E)\mathscr H$. 
\end{proof}

\begin{proposition} \label{prop:c.3} 
Let $(\pi, \mathscr H)$ be a  unitary representation 
of the Lie group $G$ and $X \in \germ g$ such that the group 
$e^{\mathbb R\mathrm{ad} (X)}$ 
preserves a norm $\|\cdot\|$ on $\germ g^\mathbb C$. 
If $
P : \germ B(\mathbb R) \to \mathcal L(\mathscr H)$
is
the spectral measure 
of the unitary one-parameter group $\pi_X(t) = \pi(\exp(tX))$
then for every open subset $E \subseteq \mathbb R$ the subspace
$(P(E) \mathscr H) \cap \mathscr H^{\omega,r}$ is dense in $P(E)\mathscr H^{\omega,r}$. 
\end{proposition}  

\begin{proof} On $\mathscr H^{\omega,r}$ we consider the Fr\'echet topology 
defined by the seminorms $(p_s)_{s < r}$ 
in Lemma~\ref{lem:2.1x}. Applying Lemma~\ref{lem:2.2}
to $K = \exp(\mathbb R X)$ implies that
all of these seminorms are invariant under $\pi_X(\mathbb R)$ 
and $\pi_X$ defines a continuous representation 
of $\mathbb R$ on $\mathscr H^{\omega,r}$ which integrates to a representation 
\[ \tilde\pi_X : (L^1(\mathbb R,\mathbb C), *) \to 
\mathrm{End}_\mathbb C(\mathscr H^{\omega,r})
\]
of the convolution algebra
that is given by
\[
\tilde\pi_X(f) = \int_\mathbb R f(t)\pi_X(t)\, dt. 
\]
This essentially means that the operators 
$\tilde\pi_X(f)$ of the integrated representation 
$L^1(\mathbb R) \to \mathcal L(\mathscr H)$ preserve the subspace 
$\mathscr H^{\omega,r}$. 

Next we write the open set $E$ as the union of the compact subsets 
\[
E_n := \Big\{ t \in E \Big| |t| \leq n, \mathrm{dist}(t,E^c) \geq \frac{1}{n}\Big\} 
\] 
and observe that $\bigcup_n P(E_n) \mathscr H$ is dense in $P(E)\mathscr H$.  
For every $n$, there exists a compactly supported function 
$h_n \in C^\infty_c(\mathbb R,\mathbb R)$ such that $\mathrm{supp}(h_n) \subseteq E$, 
$0 \leq h_n \leq 1$, and $h_n\big|_{E_n} = 1$. Let $f_n \in \mathscr S(\mathbb R)$ with 
$\hat f_n = h_n$. Then 
\[ \tilde\pi_X(f_n) = \hat f_n(-i\mathsf{d}\pi(X)) = h_n(-i\mathsf{d}\pi(X)) \] 
and consequently
\[ P(E_n)\mathscr H \subseteq \tilde\pi_X(f_n)\mathscr H \subseteq 
P(E)\mathscr H. \]
Therefore the subspace 
$\tilde\pi_X(f_n)\mathscr H^{\omega,r}$ of $\mathscr H^{\omega,r}$ is contained 
in $P(E)\mathscr H$. 
If $w = P(E)v$ for some $v \in \mathscr H^{\omega,r}$
then 
\[ \tilde\pi_X(f_n)w  = \tilde\pi_X(f_n)P(E)v = \tilde\pi_X(f_n)v
\in \mathscr H^{\omega,r}\] 
and 
\[ \|\tilde\pi_X(f_n)w-w\|^2 
=  \|h_n(-i\mathsf{d}\pi(X))w -w\|^2 \leq \|P(E\backslash E_n)w\|^2 \to 0\] 
from which it follows that $\tilde\pi_X(f_n)w \to w$. 
\end{proof}

\begin{proposition} \label{prop:4.3} If $Y \in \germ g^\mathbb C$ satisfies 
$[X,Y] = i\mu Y$ then for every open subset 
$E \subseteq \mathbb R$ the spectral measure of $\pi_X$ satisfies 
\begin{equation}
  \label{eq:shift}
\mathsf{d}\pi(Y)\big(P(E)\mathscr H  \cap \mathscr H^\infty\big) 
\subseteq P(E+ \mu)\mathscr H. 
\end{equation}
\end{proposition}

\begin{proof}
To verify this relation, we first observe that 
\[ \pi_X(t) \mathsf{d}\pi(Y)v 
= \mathsf{d}\pi(e^{t \mathrm{ad} X}(Y))\pi_X(t)v 
= e^{it\mu} \mathsf{d}\pi(Y)\pi_X(t)v  \]
for every $v \in \mathscr H^\infty$. 
For $f \in \mathscr S(\mathbb R)$, the continuity of the map 
\[ \mathscr S(\mathbb R) \to \mathscr H^\infty, \quad f \mapsto \tilde\pi_X(f)v\]
leads to 
\[ \tilde\pi_X(f)\mathsf{d}\pi(Y)v 
= \mathsf{d}\pi(Y) \int_\mathbb R f(t) e^{it\mu} \pi_X(t)v  
= \mathsf{d}\pi(Y) \tilde\pi_X(f \cdot e_\mu)v
\]
where $e_\mu(t) = e^{it\mu}$. If $v \in P(E)\mathscr H$ 
and $\hat f$ vanishes on $E+\mu$ then the function
$(e_\mu f)\,\hat{} = \hat f(\mu + \cdot) $
vanishes on $E$, and Lemma 
\ref{lem:c.1}(ii) implies that $\tilde\pi_X(f \cdot e_\mu) v = 0$.
Applying Lemma~\ref{lem:c.1}(ii) again, we 
derive that $\mathsf{d}\pi(Y)v \in P(E+\mu)\mathscr H$. 
\end{proof}

\subsection{Application to irreducible unitary representations of Lie supergroups}

Let $(\pi,\rho^\pi, \mathscr H)$ be an irreducible unitary representation of the 
Lie supergroup $\mathcal G = (G, \germ g)$. Before we turn to the fine 
structure of such a representation, we verify 
that Lemma~\ref{lem:invlem} generalizes to the super context. 

\begin{lemma} \label{lem:invlem2} The subspace $\mathscr H^{\omega,r} 
\subeq \cH$ is invariant 
under $\mathscr U(\germ g^\mathbb C)$. 
\end{lemma}

\begin{proof} In view of Lemma~\ref{lem:invlem}, it only remains 
to show that, for every $Y \in \g_\ood$ and 
$v \in \cH^{\omega,r}$, we have 
$\rho^\pi(Y)v \in \cH^{\omega,r}$. For every 
$X \in \g_\eev \cap B_r$, we have the relation 
\begin{equation}
  \label{eq:comrel}
\pi(\exp X)\rho^\pi(Y)v 
= \rho^\pi(e^{\ad X}Y)\pi(\exp X)v 
= \rho^\pi(e^{\ad X}Y)f_v(X).
\end{equation}
The complex bilinear map 
\[ \g_\ood^\C \times \cH^\infty \to \cH^\infty, \quad 
(Z,v) \mapsto \rho^\pi(Z)v\] 
is continuous by Lemma~\ref{lem:4.2.1} and therefore holomorphic. 
Moreover, the map 
\[\g_\eev^\C \to \g_\ood^\C\quad,\quad X \mapsto e^{\ad X}Y
\] 
is holomorphic. Since compositions of holomorphic maps 
are holomorphic, it therefore suffices to show that 
$f_v(B_r) \subeq \cH^\infty$ and that the map $f_v \: B_r \to \cH^\infty$ 
is holomorphic. In fact, this implies that 
the map 
\[ \g_\eev \cap B_r \to \cH, \quad X \mapsto \pi(\exp X) \rho^\pi(Y) v\] 
extends holomorphically to $B_r$, i.e., 
$\rho^\pi(Y)v \in \cH^{\omega, r}$. 

We recall the topological isomorphism 
\[ \ev_0 \: \Hol(B_r, \cH)^\g \to \cH^{\omega,r}, \quad 
f \mapsto f(0).\]  
By definition of $\Hol(B_r, \cH)^\g$, we have for each 
$X \in \g_\eev$ the relation 
\[ \dd\pi(X) \circ f_v = R_X^* f_v,\] 
showing in particular that $\dd\pi(X) \circ f_v \: B_r \to \cH$ 
is a holomorphic function. From the definition of the 
topology on $\cH^\infty$, it therefore follows that 
$f_v$ is holomorphic as a map $B_r \to \cH^\infty$. 
\end{proof}

The following theorem clarifies the key features of the 
$\g$-representation on~$\cH^\infty$. 

\begin{theorem} \label{thm:hiweistr}
Let $(\pi,\rho^\pi, \mathscr H)$ be an irreducible 
unitary representation of the 
Lie supergroup $\mathcal G = (G, \germ g)$ 
which is \STAR -reduced and 
satisfies 
\[ \germ g_\eev = [\germ g_\ood, \germ g_\ood].\] 
Pick a regular element $X_0 \in \INT(\CONE(\mathcal G))$ 
and let $\ft=\ft_\eev\oplus\ft_\ood$ be the corresponding Cartan subsuperalgebra of 
$\g$ (see Lemma \ref{lemma-compact-cartan} and Proposition 
\ref{proppenkovserganova}). Suppose that no root vanishes on $X_0$. 
Then the following assertions hold.  
\begin{description}
\item[\rm(i)] $\ft_\eev$ is compactly embedded and 
$\Delta^+ = \{\ \alpha \in \Delta\ |\ \alpha(X_0) > 0\ \}$ 
satisfies $\Delta\backslash \{0\} = \Delta^+ \dot\cup - \Delta^+.$ 
\item[\rm(ii)] The space $\cH^{\ft}$ of $\ft$-finite elements 
in $\cH^\infty$ is an irreducible $\g$-module 
which is a $\ft_\eev$-weight module and dense in $\cH$. 
\item[\rm(iii)] The maximal eigenspace $\sV$ of $i\rho^\pi(X_0)$ 
is an irreducible finite dimensional $\ft$-module on which 
$\ft_\eev$ acts by some weight $\lambda \in \ft_\eev^*$. 
It generates the $\g$-module $\cH^\ft$ and all 
other $\ft_\eev$-weights in this space are of the form 
\[ \lambda - m_1\alpha_1 - \cdots - m_k\alpha_k, \quad 
\alpha_j \in \Delta^+,\ \  k\in\N,\,m_1,\ldots,m_k \in \N\cup \{0\}.\] 
\item[\rm(iv)] Two representations 
$(\pi,\rho^\pi, \mathscr H)$ and 
$(\pi',\rho^{\pi'}, \mathscr H')$ of $\cG$ are 
isomorphic 
if and only if the corresponding $\ft$-representations 
on $\sV$ and $\sV'$ are isomorphic. 
\end{description}
\end{theorem}

\begin{proof} (i) Proposition~\ref{propertiesofstarreduced} implies that
$\CONE(\mathcal G)$ is a pointed generating invariant cone and 
$\germ g_\eev$ has a Cartan subalgebra $\germ t_\eev$ which is 
compactly embedded in $\germ g$. Then 
the corresponding Cartan supersubalgebra is given by its centralizer 
$\ft = \sZ_\g(\ft_\eev)$. 
Pick a regular element $X_0 \in \germ t_\eev \cap \INT(\CONE(\mathcal G))$, 
so that $\Delta^+$ satisfies 
$\Delta\backslash \{0\} = \Delta^+ \dot\cup - \Delta^+.$ 

(ii) Recall from \eqref{eq:dissi} that $i \rho^\pi(X_0) \leq 0$. 
We want to prove the existence of an eigenvector of maximal 
eigenvalue for $i\rho^\pi(X_0)$. 
Let 
\[\delta = \min \{  \alpha(X_0) | \alpha \in \Delta^+\} \]
and note that $\delta > 0$. 
Let $P([a,b])$, $a\leq b \in \mathbb R$, 
denote the spectral projections of the selfadjoint operator 
$i\overline{\rho^\pi(X_0)}$ and put 
\[ \lambda = \sup(\mathrm{Spec}(i\overline{\rho^\pi(X_0)})) \leq 0.\] 

Since $(\pi,\rho^\pi,\mathscr H)$ is irreducible and the space 
$\mathscr H^\omega$ of analytic vectors is dense, there exists 
an $r > 0$ with $\mathscr H^{\omega,r} \neq\{0\}$.  
Then the invariance of $\mathscr H^{\omega,r}$ under 
$\mathscr U(\germ g^\mathbb C)$ (Lemma~\ref{lem:invlem2}) 
implies that $\mathscr H^{\omega,r}$ 
is dense in~$\mathscr H$. 
Hence Proposition~\ref{prop:c.3} implies that, 
for every $\eps > 0$, the intersection 
\[ P(]\mu - \eps,\mu]) \mathscr H \cap \mathscr H^{\omega, r}\] 
is dense in $P(]\mu - \eps,\mu]) \mathscr H$. 
In particular, it contains a non-zero vector 
$v_0$. We then obtain 
with Proposition~\ref{prop:4.3} for $\eps < \delta$ 
and $\alpha \in \Delta_+$: 
\[  \rho(\germ g^{\mathbb C,\alpha}) v_0 \subseteq P(]\mu, \infty[)\mathscr H = \{0\}.\] 
In view of the Poincar\'e--Birkhoff--Witt Theorem, this leads to  
\[ \mathscr U(\germ g^\C)v_0 = \mathscr U(\germ g^- \rtimes 
\germ t^\mathbb C)v_0.\] 
Since $\germ t^\mathbb C$ commutes with $\germ t_\eev$, the subspace 
$\mathscr U(\germ t^\mathbb C)v_0$ 
is contained in \\ $P([\mu-\eps,\mu])\mathscr H$, so that 
Proposition~\ref{prop:4.3} yields 
\[ \mathscr U(\germ g^\mathbb C)v_0 \subseteq 
P(]-\infty,\mu-\delta])\mathscr H + P([\mu - \eps, \mu])\mathscr H \] 
for every $\eps > 0$. As $\mathscr U(\germ g^\mathbb C)v_0$ is dense in 
$\mathscr H$, we obtain 
for every $\eps > 0$ the relation $P([\mu - \eps, \mu]) = P(\{\mu\})$. 
Hence $i\rho(X_0)v_0 = \mu v_0$. Since 
$\germ g^\mathbb C$ is spanned by $\mathrm{ad}(X_0)$-eigenvectors, the same holds for 
$\mathscr U(\germ g^\mathbb C)$, and hence for 
$\mathscr U(\germ g^\mathbb C)v_0$. This means that $i\rho^\pi(X_0)$ is diagonalizable. 
Repeating the same argument for other regular elements  in 
$\germ t \cap \INT(\CONE(\mathcal G))$ forming a basis of $\germ t$, 
we conclude that $\rho^\pi(\germ t)$ is diagonalizable, i.e., that 
$\mathscr H$ is the orthogonal direct sum of weight spaces for 
$\germ t$, resp., the corresponding group~$T$. 

Let $\mathscr V = P(\{\mu\})\mathscr H$ be the maximal eigenspace of
$i\rho^\pi(X_0)$. Then Proposition~\ref{prop:c.3} applied to sets of the 
form $E = ]\mu-\eps, \mu+\eps[$ implies that 
$\mathscr H^{\omega,r} \cap \mathscr V$ is dense in $\mathscr V$. 
Further $\mathscr V$ is $T$-invariant, hence 
an orthogonal direct sum of $T$-weight spaces. 
From Lemma~\ref{lem:2.2}, applied to $K = T$, we now derive that 
in each $T$-weight space $\mathscr V^\alpha(T)$, the intersection 
with $\mathscr H^{\omega,r}$ is dense. 

Let $v_\alpha \in \mathscr V^\alpha \cap \mathscr H^{\omega,r}$ be a 
$T$-eigenvector. From the density of 
$\mathscr U(\germ g^\mathbb C)v_\alpha = \mathscr U(\germ g^-)
\mathscr U(\germ t^\mathbb C)v_\alpha$ 
in $\mathscr H$ we then derive as above that 
\[ \mathscr U(\germ t^\mathbb C)v_\alpha = 
\mathscr U(\germ t_\ood^\mathbb C)v_\alpha \subseteq \mathscr V^\alpha \] 
is dense in $\mathscr V$. As $\mathscr U(\germ t_{\overline 1}^\mathbb C)$ 
is finite dimensional, 
this proves that $\mathscr V = \mathscr V^\alpha$ is finite dimensional and contained in $\mathscr H^{\omega, r}$. 

Since all $\ft_\eev$-weight spaces in $\sU(\g^-)$ are finite 
dimensional and $\sU(\ft_\ood)$ is finite dimensional, we conclude that 
$\sU(\g^\C)\sV$ is a locally finite $\ft$-module with 
finite $\ft_\eev$-multiplicities. In view of the finite multiplicities, 
its density in $\cH$ leads to the equality $\cH^\ft = \sU(\g^\C) \sV$. 
As this $\g$-module consists of analytic vectors, its irreducibility 
follows from the irreducibility of the $\cG$-representation 
on $\cH$.  

(iii) If $\mathscr V' \subseteq \mathscr V$ is a non-zero $\germ t$-submodule, then 
$\mathscr U(\germ t)\mathscr V'$ is dense in $\mathscr V$ and orthogonal to 
the subspace $\mathscr V'' = (\sV')^\bot$, 
which leads to $\mathscr V'' = \{0\}$. Therefore the $\germ t$-module
$\mathscr V$ is irreducible. All other assertions have already been 
verified above. 

(iv) Clearly, the equivalence of the $\cG$-representations 
implies equivalence of the $\ft$-representations on $\sV$ and $\sV'$. 

Suppose, conversely, that there exists a 
$\ft$-isomorphism $\phi \: \sV \to \sV'$. 
We consider the direct 
sum representation $\cK = {\cal H} \oplus {\cal H}'$ of $\cG$, 
for which 
\[ \cK^\ft = \cH^\ft \oplus (\cH')^\ft \] 
as $\g$-modules. Consider the 
$\g$-submodule $W \subeq \cK^\ft$ generated by the 
$\ft$-submodule 
\[ \Gamma(\phi) = \{ (v, \phi(v)) \: v \in \sV \} 
\subeq \sV \oplus \sV'.\] 
Since $\Gamma(\phi)$ is annihilated by 
$\g^+$, the PBW Theorem implies that 
\[ W 
= \sU(\g)\Gamma(\phi) = \sU(\g^-)\sU(\ft)\sU(\g^+) \Gamma(\phi) 
= \sU(\g^-)\Gamma(\phi).\] 
It follows that 
\[ W \cap (\sV \oplus \sV') = \Gamma(\phi)\] 
is the maximal eigenspace for $iX_0$ on $W$. 

As $W$ consists of analytic vectors, 
its closure $\overline W$ is a proper $G$-invariant subspace of $\cK$, so 
that we obtain a unitary $\cG$-representation on this space. 

If the two $\cG$ representations 
$(\pi,\rho^\pi, \mathscr H)$ and 
$(\pi',\rho^{\pi'}, \mathscr H')$ are not equivalent, 
then Schur's Lemma implies that 
$\cH$ and $\cH'$ are the only non-trivial 
$\cG$-invariant subspaces of $\cK$, contradicting the 
existence of $\overline W$. 
\end{proof}

\begin{remark}  (a) The preceding theorem suggests to call 
the $\g$-representation on $\cH^\ft$ a {\it highest weight representation} 
because it is generalized by a weight space space of 
$\ft_\eev$ which is an irreducible $\ft$-module, hence a 
(finite dimensional) irreducible 
module of the Clifford Lie superalgebra 
$\ft_\ood + [\ft_\ood, \ft_\ood]$. 

(b) Suppose that $\g$ is $\star$-reduced with 
$\g_\eev = [\g_\ood,\g_\ood]$. 
Let $\cH$ be a complex Hilbert space and 
$\cD \subeq \cH$ a dense subspace on which we have a 
unitary representation $(\rho,\cD)$ of $\g$ in the sense that 
(i), (iii), (v) in Definition~\ref{defofunirep} are 
satisfied. 

Suppose further that the action of 
$\ft_\eev$ on $\cD$ is diagonalizable with finite dimensional 
weight spaces. Then the $\g$-module $\cD$ is semisimple, 
hence irreducible if it is generated by a $\ft_\eev$-weight space 
$V$ on which $\ft$ acts irreducibly. 


The finite dimensionality of the $\ft_\eev$-weight spaces 
on $\cD$ also implies the semisimplicity of $\cD$ as a 
$\g_\eev$-module. Hence, as $i\rho(X_0) \leq 0$, 
an argument as in the proof of Theorem~\ref{thm:hiweistr} 
implies that each simple submodule of $\cD$ is a unitary highest weight module, 
hence integrable by \cite[Cor.~XII.2.7]{Ne00}. 
We conclude in particular that the $\g_0$-representation on $\cD$  
is integrable with $\cD$ consisting of analytic vectors. 
\end{remark}


\section{The orbit method and nilpotent Lie supergroups}\label{sec:8}

One of the most elegant and powerful ideas in the theory of unitary representations of 
Lie groups since the early stages of its development is the \emph{orbit method}. 
The basic idea of the orbit method is to attach unitary representations to special homogeneous 
symplectic manifolds, such as the coadjoint orbits,  in a natural way. One of the goals of the 
orbit method is to obtain a concrete realization of the representation and to extract 
information about the representation (e.g., its distribution character)
from this realization.

Recall that a Lie supergroup $\mathcal G=(G,\germ g)$ is called \emph{nilpotent} when the 
Lie superalgebra $\germ g$ is nilpotent.
In this article the orbit method is only studied for nilpotent Lie supergroups.
It is known that among Lie groups, the orbit method works best for the class of nilpotent ones.
For further reading on the subject of the orbit method, the reader is
referred to \cite{kirillovbook} and \cite{voganarticle}.

\subsection{Quantization and polarizing subalgebras}
\label{classical-orbit-method}
All of the irreducible unitary representations
of nilpotent Lie groups can be classified by the orbit method. 
Let $G$ be a nilpotent real Lie group and $\germ g$ be its Lie algebra.
For simplicity, $G$ is assumed to be simply connected. In this case, there exists a bijective 
correspondence between coadjoint orbits (i.e., $G$-orbits in $\germ g^*$) and irreducible unitary
representations of $G$. In some sense the correspondence is surprisingly simple. To construct a 
representation $\pi_\mathcal O$ of $G$ which corresponds to a coadjoint orbit 
$\mathcal O\subseteq\germ g^*$, one first chooses 
an element $\lambda\in\mathcal O$ and considers the skew symmetric form 
\begin{equation}
\label{omega-equation}
\Omega_\lambda:\germ g\times\germ g\to\mathbb R
\textrm{\ \  defined by\ \  }\Omega_\lambda(X,Y)=\lambda([X,Y]).
\end{equation}
It can be shown that there exist maximal isotropic subspaces of 
$\Omega_\lambda$ which are also subalgebras of $\germ g$.
Such subalgebras are called \emph{polarizing subalgebras}. 
For a given polarizing subalgebra $\germ m$ of $\germ g$, one 
can consider the one dimensional representation of 
the subgroup $M=\exp(\germ m)$ of $G$ given by
\[
\chi_\lambda(m)=e^{i\lambda(\log(m))}\quad\textrm{ for }m\in M.
\]
The unitary representation of $G$ corresponding to $\mathcal O$ 
is $\pi_\mathcal O=\IND_M^G\chi_\lambda$. Of course one needs to prove that the construction
is independent of the choices of $\lambda$ and $\germ m$, the representation
$\pi_\mathcal O$ is irreducible, and the correspondence is bijective.
These statements are proved in \cite{kirillovpaper}. Many other proofs have been found as well.

\subsection{Heisenberg--Clifford Lie supergroups} Heisenberg groups play a distinguished role in the harmonic analysis of nilpotent 
Lie groups. Therefore 
it is
natural to expect that 
the analogues of Heisenberg groups in the category of Lie supergroups play a similar role 
in the representation theory of
nilpotent Lie supergroups. These analogues, which deserve to be called \emph{Heisenberg--Clifford} Lie supergroups,
can be described as follows. Let $(\mathsf W,\mathsf\Omega)$ be a finite dimensional real \emph{super symplectic} vector space. This means that
$\mathsf W=\mathsf W_\eev\oplus\mathsf  W_\ood$ is endowed with a bilinear form
\[
\mathsf\Omega:\mathsf W\times \mathsf W\to\mathbb R
\]
that satisfies the following properties.
\begin{description}[iiiiii]
\item[\rm(i)] $\mathsf\Omega(\mathsf W_\eev,\mathsf W_\ood)=
\mathsf\Omega(\mathsf W_\ood,\mathsf W_\eev)=\{0\}$.
\item[\rm(ii)] The restriction of $\mathsf\Omega$ to $\mathsf W_\eev$ is a symplectic form.
\item[\rm(iii)] The restriction of $\mathsf\Omega$ to $\mathsf W_\ood$ is a nondegenerate symmetric form.
\end{description}
The Heisenberg--Clifford Lie supergroup corresponding to $(\mathsf W,\mathsf \Omega)$ is 
the super Harish--Chandra pair $(H^\mathsf W,\germ h^\mathsf W)$ where

\begin{description}[iiiiii]
\item[\rm(i)] $\displaystyle\germ h^\mathsf W_\eev=\mathsf W_\eev\oplus \mathbb R$ and 
$\displaystyle \germ h^\mathsf W_\ood=\mathsf W_\ood\,$ (as vector spaces).
\item[\rm(ii)] for every $X,Y\in\mathsf W$ and every $a,b\in\mathbb R$,
the superbracket of $\germ h^\mathsf W$ is defined by
\[
[(X,a),(Y,b)]=(0,\mathsf\Omega(X,Y)).
\] 
\item[\rm(iii)] $H^{\mathsf W}$ is the simply connected Lie group 
with Lie algebra $\germ h^\mathsf W_\eev$.
\end{description}

When $\dim\mathsf W_\ood=0$ the Lie supergroup $(H^\mathsf W,\germ h^\mathsf W)$ is purely even, 
i.e., it is a Lie group. In this case, it is usually called a
\emph{Heisenberg Lie group}. When $\dim\mathsf W_\eev=0$ the Lie supergroup
 $(H^\mathsf W,\germ h^\mathsf W)$ is  
called a \emph{Clifford Lie supergroup}.

Irreducible unitary representations of Heisenberg Lie groups are
quite easy to classify. One can use the orbit method 
of Section \ref{classical-orbit-method} to classify them, but their classification was known
as a consequence of the Stone--von Neumann Theorem long before
the orbit method was developed. The Stone--von Neumann Theorem implies that
there exists a bijective correspondence between infinite dimensional 
irreducible unitary representations of a Heisenberg Lie group and nontrivial 
characters (i.e., one dimensional unitary representations) of its center.

For Heisenberg--Clifford Lie supergroups there is a similar classification of representations.
Let $(\pi,\rho^\pi,\mathscr H)$ be an irreducible unitary representation of 
$(H^\mathsf W,\germ h^\mathsf W)$.
By a super version of Schur's Lemma, for every $Z\in\mathscr Z(\germ h^\mathsf W)$ the action of 
$\rho^\pi(Z)$ is via multiplication by a scalar $c_{\rho^\pi}(Z)$.
If $c_{\rho^\pi}(Z)=0$ for every $Z\in\mathscr Z(\germ h^\mathsf W)$, then
$\mathscr H$ is one dimensional, and essentially obtained from a unitary character of $\mathsf W_\eev$. 
The irreducible unitary representations for which $\rho^\pi(\mathscr Z(\germ h^\mathsf W))\neq \{0\}$ 
are classified by the following statement 
(see \cite[Thm.~5.2.1]{salmasian}).

\begin{theorem}
\label{heiscliff}
Let $\mathsf S$ be the set of unitary equivalence classes of
irreducible unitary representations $(\pi,\rho^\pi,\mathscr H)$  
of $(H^\mathsf W,\germ h^\mathsf W)$ for which
$\rho^\pi(\mathscr Z(\germ h^\mathsf W))\neq \{0\}$. 
Then $\mathsf S$ is nonempty if and only if 
the restriction of $\mathsf\Omega$ to $\mathsf W_\ood$ is (positive or negative) definite.
Moreover, 
the map 
\[
[(\pi,\rho^\pi,\mathscr H)]\mapsto c_{\rho^\pi}
\]
yields a surjection from $\mathsf S$
onto the set of $\mathbb R$-linear functionals $\gamma:\mathscr Z(\germ h^\mathsf W)\to\mathbb R$ 
which satisfy
\[
i\gamma([X,X])<0\textrm{ for every } 0 \not= X\in\mathsf W_\ood.
\]
When $\dim\mathsf W_\ood$ is odd the latter map is a bijection, and when $\dim\mathsf W_\ood$ is even it is two-to-one, and the
two representations in the fiber are isomorphic via parity change.
\end{theorem}

Every irreducible unitary representation of a Clifford Lie supergroup is finite dimensional
(see \cite[Sec. 4.5]{salmasian}).
In fact the theory of Clifford modules implies that the only possible values for the dimension of such a representation
are one or 
\[
2^{\left(\dim \mathsf W_\ood\,-\,\lfloor\frac{\dim\mathsf W_\ood}{2}\rfloor\right)}.
\] 
It will be seen below that Clifford Lie supergroups are used to define analogues of polarizing subalgebras for Lie supergroups.

\subsection{Polarizing systems and a construction}
\label{aconstruction}

In order to construct the irreducible unitary representations of a nilpotent Lie supergroup 
using the orbit method, first we need to generalize the notion of polarizing subalgebras. 
What makes the case of Lie supergroups more complicated than the case of Lie groups is
the fact that irreducible unitary representations of nilpotent Lie supergroups
are not necessarily induced from one dimensional representations. However, it will be seen that
they are induced from certain finite dimensional representations
which are obtained from representations of Clifford Lie supergroups.

Let $(G,\germ g)$ be a Lie supergroup. Associated to
every $\lambda\in\germ g_\eev^*$ there exists a 
skew symmetric bilinear form $\Omega_\lambda$ on $\germ g_\eev$ which 
is defined in \eqref{omega-equation}.
There is also a symmetric bilinear form 
\begin{equation}
\label{defofomegalambda}
\mathsf\Omega_\lambda:\germ g_\ood\times\germ g_\ood\to\mathbb R
\end{equation}
associated to $\lambda$, which is defined by
$\mathsf\Omega_\lambda(X,Y)=\lambda([X,Y])$.

\begin{definition}
\label{defofpolarizingsystem}
Let $\mathcal G=(G,\germ g)$ be a nilpotent Lie supergroup. A polarizing system in $(G,\germ g)$ is a 
pair $(\mathcal M,\lambda)$ satisfying the following properties.
\begin{description}[iiiiii]
\item[\rm(i)] $\lambda\in\germ g_\eev^*$ and $\mathsf\Omega_\lambda$ is a positive semidefinite form.
\item[\rm(ii)] $\mathcal M=(M,\germ m)$ is a 
Lie subsupergroup of $\mathcal G$ and $\dim\germ m_\ood=\dim\germ g_\ood$.
\item[\rm(iii)] $\germ m_\eev$ is a polarizing subalgebra of $\germ g_\eev$ with respect to $\lambda$, i.e.,
a subalgebra of $\germ g_\eev$ which is also a maximal isotropic subspace with respect to $\Omega_\lambda$.
\end{description}
\end{definition}
Given a polarizing system $(\mathcal M,\lambda)$, one can construct a unitary representation of 
$\mathcal G$ as follows. Let 
\begin{equation}
\label{jdefined}
\germ j=\ker\lambda\oplus\mathrm{rad}(\mathsf\Omega_\lambda)
\end{equation}
where $\mathrm{rad}(\mathsf\Omega_\lambda)$
denotes the radical of $\mathsf\Omega_\lambda$. One can show that $\germ j$ is an ideal of $\germ m$ 
that corresponds to a Lie subsupergroup $\mathcal J=(J,\germ j)$ of $\mathcal M$,
and the quotient $\mathcal M/\mathcal J$ is a Clifford Lie supergroup.
Let $\mathscr Z(\germ m/\germ j)$ denote the center of $\germ m/\germ j$.
Since $\mathsf\Omega_\lambda$ is positive semidefinite, from Theorem \ref{heiscliff} it follows that
up to parity and unitary equivalence there exists a unique unitary representation 
$(\sigma,\rho^\sigma,\mathscr K)$ 
of $\mathcal M/\mathcal J$ such that 
for every $Z\in \mathscr Z(\germ m/\germ j)$, the operator $\rho^\sigma(Z)$ acts via multiplication by $i\lambda(Z)$.
Clearly $(\sigma,\rho^\sigma,\mathscr K)$ can also be thought of as a representation of $\mathcal M$, and
one can consider the induced representation
\begin{equation}
\label{piinduction}
(\pi,\rho^\pi,\mathscr H)=\IND_\mathcal M^\mathcal G(\sigma,\rho^\sigma,\mathscr K).
\end{equation}

\subsection{Existence of polarizing systems}
\label{existenceofpolarization}
Throughout this section $\mathcal G=(G,\germ g)$ will be a nilpotent Lie supergroup 
such that $G$ is simply connected.

It is natrual to ask for which
$\lambda\in\germ g_\eev^*$ such that  
$\mathsf\Omega_\lambda$ is positive semidefinite there exists a polarizing system 
$(\mathcal M,\lambda)$ in the sense of Definition \ref{defofpolarizingsystem}.
It turns out that for all such $\lambda$ the answer is affirmative. 
The latter statement can be proved as follows. 
Fix such a $\lambda\in\germ g_\eev^*$. Proving the existence of a polarizing system $(\mathcal M,\mathcal \lambda)$ 
amounts to showing that there exists a polarizing subalgebra $\germ m_\eev$ of $\germ g_\eev$ such that 
$\germ m_\eev\supseteq[\germ g_\ood,\germ g_\ood]$. Since $[\germ g_\ood,\germ g_\ood]$ is an ideal of
the Lie algebra $\germ g_\eev$ and $\germ g_\eev$ is nilpotent, one can find a sequence of ideal of $\germ g_\eev$
such as 
\[
\{0\}=\germ i^{(0)}\subseteq\germ i^{(1)}\subseteq\cdots\subseteq\germ i^{(k-1)}\subseteq
\germ i^{(k)}=
[\germ g_\ood,\germ g_\ood]
\subseteq\germ i^{k+1}\subseteq\cdots\subseteq\germ i^{(r)}=\germ g_\eev
\]
where for every $0\leq s\leq r-1$, the codimension of  $\germ i^{(s)}$ in $\germ i^{(s+1)}$ is equal to one.
For every $0\leq s\leq r-1$, let 
\[
\Omega_{\lambda}^{(s)}:\germ i^{(s)}\times\germ i^{(s)}\to\mathbb R
\]
be the skew symmetric form 
defined by 
\[
\Omega_{\lambda}^{(s)}(X,Y)=\lambda([X,Y])
\] and let $\mathrm{rad}(\Omega_\lambda^{(s)})$ denote the radical of 
$\Omega_\lambda^{(s)}$.
It is known that the subspace of $\germ g_\eev$ defined by
\[
\mathrm{rad}(\Omega_\lambda^{(1)})+\cdots+\mathrm{rad}(\Omega_\lambda^{(s)})
\]
is a polarizing subspace of $\germ g_\eev$ corresponding to $\lambda$ (see \cite[Th. 1.3.5]{vergnelemmacorwingreenleaf}).
To prove that
\[
[\germ g_\ood,\germ g_\ood]\subseteq\mathrm{rad}(\Omega_\lambda^{(1)})+\cdots+\mathrm{rad}(\Omega_\lambda^{(s)})
\]
it suffices to show that $\mathrm{rad}(\Omega_\lambda^{(k)})=\germ i^{(k)}$.
This is where one needs the fact that $\mathsf\Omega_\lambda$ is positive semidefinite.
The proof is by a backward induction on the dimension of $\mathcal G$. Details of the argument appear in
\cite[Sec. 6.3]{salmasian}.

\subsection{A bijective correspondence}
\label{bijectivesection}
Throughout this section $\mathcal G=(G,\germ g)$ will be a nilpotent Lie supergroup 
such that $G$ is simply connected.

One can check easily that the set
\[
\mathscr P(\mathcal G)=\big\{\ \lambda\in\germ g_\eev^*\ \big|\ \mathsf\Omega_\lambda\textrm{ is positive semidefinite}\ \big\}
\]
is an invariant cone in $\germ g_\eev^*$.
Section \ref{existenceofpolarization} shows that for every 
$\lambda\in\mathscr P(\mathcal G)$ 
one can find
a polarizing system $(\mathcal M,\lambda)$. Therefore the construction of Section \ref{aconstruction} 
yields a unitary representation $(\pi_\lambda,\rho^{\pi_\lambda},\mathscr H_\lambda)$ of 
$\mathcal G$ which is given by \eqref{piinduction}.
The main result of \cite{salmasian} can be stated as follows.
\begin{theorem}
\label{mainnilp}
The map which takes a $\lambda\in\mathscr P(\mathcal G)$ to the representation 
$(\pi_\lambda,\rho^{\pi_\lambda},\mathscr H_\lambda)$ 
results in a bijective correspondence between $G$-orbits in 
$\mathscr P(\mathcal G)$ and irreducible unitary representations of $\mathcal G$ up to unitary equivalence and 
parity change.
\end{theorem}
To prove Theorem \ref{mainnilp} one needs to show that the construction given in Section \ref{aconstruction} 
yields an irreducible
representation and is independent of the choice of $\lambda$ in a $G$-orbit or the polarizing system. One also has to show that
if $\lambda$ and $\lambda'$ are not in the same $G$-orbit then inducing from 
polarizing systems $(\mathcal M,\lambda)$ and $(\mathcal M',\lambda')$
does not lead to representations which are identical up to parity or unitary equivalence. The proofs of all of these facts are 
given in \cite[Sec. 6]{salmasian}. To some extent, the method of proof is similar to the original proof of the Lie group 
case in \cite{kirillovpaper}, where induction on the dimension is used. In the Lie group case, what makes the inductive argument work 
is the existence of three dimensional Heisenberg subgroups in any nilpotent Lie group
of dimension bigger than one with one dimensional center.  
For Lie supergroups a similar statement only holds under extra assumptions.
The next proposition shows that it suffices to assume that the corresponding 
Lie superalgebra has no self-commuting odd elements. 
\begin{proposition}
\label{kirillovlmma}
Let $\mathcal G=(G,\germ g)$ be as above. Assume that
there are no nonzero $X\in\germ g_\ood$
such that $[X,X]=0$. If $\dim \mathscr Z(\germ g)=1$
then either $\mathcal G$ is a Clifford Lie supergroup, 
or it has a Heisenberg Lie subsupergroup of
dimension $(3|0)$.
\end{proposition}
Using 
Proposition \ref{reducedlemma} one can pass to a quotient and reduce the analysis of
the general case to the case 
where the assumptions of Proposition \ref{kirillovlmma} are satisfied. 
Proposition \ref{kirillovlmma}  makes induction on the dimension of 
$\germ g$ possible. 

Although the proof of Theorem \ref{mainnilp} 
is insipred by the methods and arguments in \cite{kirillovpaper} 
and \cite{vergnelemmacorwingreenleaf},
one must tackle numerous additional analytic technical difficulties which
emerge in the case of Lie supergroups. This is because many facts in the theory of unitary representations 
of Lie supergroups are generally not as powerful as 
their analogues for Lie groups. For instance to 
prove that $(\pi_\lambda,\rho^{\pi_\lambda},\mathscr H_\lambda)$ is irreducible one 
cannot use Mackey theory and needs new ideas.

\subsection{Branching to the even part}
Let $\mathcal G=(G,\germ g)$ be as in Section \ref{bijectivesection}.
For every $\lambda\in\mathscr P(\mathcal G)$ let  
$(\pi_\lambda,\rho^{\pi_\lambda},\mathscr H_\lambda)$
be the representation of $\mathcal G$ associated to $\lambda$ in Section \ref{bijectivesection}.
As an application of Theorem \ref{mainnilp} one can obtain a simple decomposition formula for the restriction of 
$(\pi_\lambda,\rho^{\pi_\lambda},\mathscr H_\lambda)$  to $G$.

Recall that  
$
(\pi_\lambda,\rho^{\pi_\lambda},\mathscr H_\lambda)
$ 
is induced from a polarizing system $(\mathcal M,\lambda)$.
Let $\germ m$ be the Lie superalgebra of $\mathcal M$ and 
$\germ j$ be defined as in \eqref{jdefined}.
\begin{corollary}
The representation $(\pi_\lambda,\mathscr H_\lambda)$
of $G$ decomposes into a direct sum of $2^{\dim\germ m-\dim\germ j}$ copies of the irreducible unitary representation
of $G$ which is associated to the coadjoint orbit containing $\lambda$ (in the sense of 
Section \ref{classical-orbit-method}). 
\end{corollary}

\section{Conclusion} 

In this note we discussed irreducible unitary representations 
of Lie supergroups in some detail for the case where 
$\cG$ is either nilpotent or $\g$ is $\star$-reduced and 
satisfies $\g_\eev =[\g_\ood, \g_\ood]$. The overlap between 
these two classes is quite small because for any nilpotent Lie superalgebra 
satisfying the latter conditions $\g_\eev$ is central, so that 
it essentially is a Clifford--Lie superalgebra, 
possibly with a multidimensional center, and in this case 
the irreducible unitary representations are the well-known 
spin representations. Precisely these representations occur 
as the $\ft$-modules on the highest weight space $\sV$ in the other 
case. 

Clearly, the condition of being $\star$-reduced is natural 
if one is interested in unitary representations. 
The requirement that $\g_\eev =[\g_\ood, \g_\ood]$ is more 
serious, as we have seen in the nilpotent case. In general 
one can consider the ideal
$\g_c = [\g_\ood, \g_\ood]\oplus\g_\ood$ 
and our results show that the irreducible unitary representations 
of this ideal are highest weight representations.
For nilpotent Lie supergroups, how to use them 
to parametrize the irreducible unitary representations 
of $\cG$ was explained in 
Section~\ref{sec:8}. It is conceivable that other larger 
classes of groups could be studied by combining 
tools from the Orbit Method, induction procedures and 
highest weight theory.

\begin{table}

\renewcommand{\arraystretch}{1.3}
\scriptsize
\begin{tabular}{|c|ll|l|}
\hline
\multicolumn{1}{|c|}{$\germ s\otimes_\mathbb R \mathbb C$} & \multicolumn{2}{c|}{ $ \germ s$} &
\multicolumn{1}{c|}{$\germ s_\eev/\mathrm{rad}(\germ s_\eev)$}\\
\hline
\multirow{3}{2cm}{\begin{center}$\mathbf A(m-1|n-1)$ \\	$m>n>1$\end{center} 	 }
					& $\germ{su}(p,q|r,s)$		& ($p+q=m,r+s=n$)		& $\germ{su}(p,q)\oplus\germ{su}(r,s)$\\
					& $\germ{su}^*(2p|2q)$ 			&  ($m = 2p$, $n= 2q$ even)	&  $\germ{su}^*(2p)\oplus\germ{su}^*(2q)$\\
					& $\germ{sl}(m|n,\mathbb R)$	&						& $\germ{sl}(m,\mathbb R)\oplus\germ{sl}(n,\mathbb R)$\\
\hline
\multirow{5}{2cm}{\begin{center}$\mathbf A(m-1|m-1)$\\	$n>1$\end{center} }
		& 	$\germ{psu}(p,q|r,s)$		& $(p+q=r+s=m)$			& $\germ{su}(p,q)\oplus\germ{su}(r,s)$\\
		& $\germ{psu}^*(2p|2p)$ 			&  ($m=2p$ even)			&  $\germ{su}^*(2p)\oplus\germ{su}^*(2p)$\\
		& $\germ{psl}(m|m,\mathbb R)$	&						& $\germ{sl}(m,\mathbb R)\oplus\germ{sl}(m,\mathbb R)$\\
		& $\germ{p\overline{q}}(m)$ 	&		  				& $\germ{sl}(m,\mathbb C)$\\
		& $\germ{usp}(m)$ 				&  						& $\germ{sl}(m,\mathbb C)$\\


\hline
\multirow{1}{2cm}[5pt]{\begin{center}$\germ{osp}(m|2n,\C)$\end{center}}  
						& $\germ{osp}(p,q|2n)$  		&  ($p+q=2m+1$)						& $\germ{so}(p,q)\oplus\germ{sp}(2n,\mathbb R)$\\
				& 	$\germ{osp}^*(m|p,q)$ 	& ($p+q=n$)	& $\germ{so}^*(m)\oplus\germ{sp}(p,q)$\\
\hline
\multirow{5}{2cm}{\begin{center}$\mathbf D(2\,|\,1,\alpha)$\\[3mm] $\alpha=\overline\alpha$ \ \ or\\ $\alpha=-1-\overline\alpha$\end{center}}
								 & $D(2|1,\alpha,2)$ 			 		&  $\alpha\in\mathbb R$ 
								 	& $\germ{sl}(2,\mathbb R)\oplus\germ{sl}(2,\mathbb R)\oplus\germ{sl}(2,\mathbb R)$\\
								 &$D(2|1,\alpha,0)$ 			 		& $\alpha\in\mathbb R$
								 	& $\germ{sl}(2,\mathbb R)\oplus\germ{su}(2)\oplus\germ{su}(2)$\\
								 &$D(2|1,\frac{1}{\alpha},0)$	 		& $\alpha\in\mathbb R$ 
								 	& $\germ{sl}(2,\mathbb R)\oplus\germ{su}(2)\oplus\germ{su}(2)$\\
								 &$D(2|1,-\frac{\alpha}{1+\alpha},0)$   & $\alpha\in\mathbb R$ 
								 	& $\germ{sl}(2,\mathbb R)\oplus\germ{su}(2)\oplus\germ{su}(2)$\\
								 &$D(2|1,\alpha,1)$  					& $\alpha=-1-\overline\alpha$
								 	& $\germ{sl}(2,\mathbb R)\oplus\germ{sl}(2,\mathbb C)$\\
\hline
\multirow{4}{2cm}{\begin{center}$\mathbf F(4)$\end{center}} 
			   & $\mathbf F(4,0)$& & $\germ{sl}(2,\mathbb R)\oplus\germ{so}(7)$\\
			   & $\mathbf F(4,1)$& & $\germ{su}(2)\oplus\germ{so}(1,6)$\\
			   & $\mathbf F(4,2)$& & $\germ{su}(2)\oplus\germ{so}(2,5)$\\
               & $\mathbf F(4,3)$& & $\germ{sl}(2,\mathbb R)\oplus\germ{so}(3,4)$\\
\hline
\multirow{2}{2cm}[5pt]{\begin{center}$\mathbf G(3)$\end{center}} 	
				& $\mathbf G(3,1)$& & $\germ{sl}(2,\mathbb R)\oplus\mathrm{Der}_\mathbb R(\mathbb O)$\\
			   	& $\mathbf G(3,2)$& & $\germ{sl}(2,\mathbb R)\oplus\mathrm{Der}_\mathbb R(
			   	\mathbb O_\mathsf{split})$\\
\hline
\multirow{2}{2cm}[5pt]{\begin{center}$\mathbf P(n-1)$\end{center}}
				 			 & $\germ{sp}(n,\mathbb R)$ 	& 				& 
				 			 $\germ{sl}(n,\mathbb R)$\\
				 			 & $\germ{sp}^*(n)$ 			& ($n$ even)   	& $\germ{su}^*(n)$\\
\hline
\multirow{3}{2cm}{\begin{center}$\mathbf Q(n-1)$\end{center}}  	
					& $\germ{psq}(n,\mathbb R)$ 				&				 & $\germ{sl}(n,\mathbb R)$\\
			  	  	& $\germ{psq}(p,q)$ 						& ($p+q=n$) 	 & $\germ{su}(p,q)$\\
			  		&  $\germ{psq}^*(n)$ 					& ($n$ even) 	 & $\germ{su}^*(n)$\\
\hline
$\mathbf W(n)$	
					 & $\mathbf W(n,\mathbb R)$& &$\germ{gl}(n,\mathbb R)$\\

\hline
$\mathbf S(n)$
						& $\mathbf S(n,\mathbb R)$ & &$\germ{sl}(n,\mathbb R)$\\
\hline
$\tilde{\mathbf S}(n)$\ \ \  $n$ even
										 & $\tilde{\mathbf S}(n,\mathbb R)$& &$\germ{sl}(n,\mathbb R)$ \\
\hline
$\mathbf H(n)$
							 & $\mathbf H(p,q)$&($p+q=n$) & $\germ{so}(p,q)$\\
							 
\hline
\end{tabular}
\vspace{1mm}

\caption{Simple real Lie superalgebras with nontrivial odd part}
\end{table}

\end{document}